\documentclass[10pt, a4paper, twoside]{article}
\usepackage{amsmath,amsthm,amssymb,bm,array,enumitem}
\usepackage{lmodern}
\renewenvironment{abstract}
{\list{}{\rightmargin=2.5cm\leftmargin=\rightmargin}%
\item[]{\sc Abstract.}\small
\relax}{\endlist}
\usepackage{fancyhdr} 
\theoremstyle{custom} 
\newtheorem{thm}{Theorem}
\newtheorem{lem}{Lemma}
\newtheorem{prop}{Proposition}
\newtheorem{cor}{Corollary}

\newtheorem*{ac*}{Acknowledgement}
\newtheorem*{rmk*}{Remark}

\usepackage[
colorlinks=true,
linkcolor=red,
anchorcolor=green,
citecolor=blue]{hyperref}

\begin{document}
\title{\bf Nonvanishing of
geodesic periods over compact hyperbolic manifolds}
\author{\upshape Feng\,\, Su}\date{}
\maketitle
\thispagestyle{empty}
\allowdisplaybreaks
\vspace*{-0.8cm}
\begin{abstract}
Let $X$ be a compact hyperbolic manifold with dimension $d\geqslant3$. In this paper we show that there are infinitely many nonvanishing geodesic periods defined over any compact $n$-dimensional ($n\geqslant2$)
 geodesic cycle of $X$.
\end{abstract}

\section{Introduction}\label{intro}
Let $X$ be a  compact hyperbolic manifold with dimension $d\geqslant 3$ and volume form $dx$, $\{\phi_j\}$ be an orthonormal basis of $L^2(X,\,dx)$ such that each $\phi_j$ is a Laplace eigenfunction: $\Delta\,\phi_j=\lambda_j\,\phi_j$ where $\lambda_j\in\mathbb{R}_{\geqslant0}$
and $\Delta$ stands for the Laplace operator of $X$ determined by its hyperbolic metric. In the theory of automorphic forms, $\phi_j$'s are called (normalized) ``Maass forms'' (after H. Maass).  
The order of $\phi_j$'s are arranged such that $\lambda_j$'s are nondecreasing as $j$ grows. Let $Y$ be a compact geodesic cycle of $X$ with dimension $n\geqslant 2$ and hyperbolic measure $dy$ (see Sect.\,\ref{hy}), 
$\psi$ be a normalized Maass form on $Y$ with Laplace eigenvalue $\lambda$. Define the {\it geodesic period} ({\it period} for short) as 
$$
P_Y(\phi_j,\psi):=\int_Y\phi_j(y)\overline{\psi}(y)dy.$$
Such period fits into the general notion of automorphic period  which has been playing a
central role in the study of
automorphic forms thanks to its close relations with 
automorphic representations and special $L$-values (see \cite{ii}, \cite{wa}, \cite{zh}, etc.).
With notations and restrictions as above, the main conclusion of this paper is 
\begin{thm}\label{thm}
For any fixed $\psi$, there are infinitely many $j$ such that $P_Y(\phi_j,\psi)\ne0$.
\end{thm}
In literature the  nonvanishing of (infinitely many) geodesic periods can follow from
the asymptotics of these
periods. See \cite{hej},
\cite{good}, \cite{ze},
\cite{p}, \cite{ts} and \cite{kmw}. In particular,  
 \cite{ts} dealt with a class of arithmetic hyperbolic manifolds and the periods defined over the codimension-1 geodesic cycles; the case where $X$ is an arbitrary
compact Riemann surface 
was treated in \cite{kmw} which has inspired us to work on higher dimensional situation.

In the language of automorphic representations,
Theorem \ref{thm} says that,
under certain restrictions on $\Gamma$,
there are infinitely many 
$G_0$-distinguished spherical automorphic representations where $G_0\subset G=SO(1,d)$
is a closed subgroup which descends to $Y$ (see Sect.\,\ref{hy}). 

The content of this paper is organized as follows. In the next section we shall make some preparations on the necessary knowledge on hyperbolic space, representation theory and trace formula. These stuff will be used later. In Sect.\,\ref{spec} we shall insert a special test function
into the trace formula and  discuss the spectral side of the trace formula. 
The analysis of the geometric side will be given in Sect.\,\ref{geo} where we split this side into two parts: the main and error terms; in particular, the contributions from these two terms will be estimated.
Theorem \ref{thm} then follows from the comparison of the spectral and geometric sides.
The necessary conditions that the test function should fulfill (so that we can apply the trace formula) will be checked in Sect.\,\ref{fkf}.

\section{Preliminaries}\label{pre}
\subsection{Hyperbolic spaces}\label{hy}
By the uniformization theorem, any $d$-dimensional orientable hyperbolic manifold $X$ with finite volume can be realized as a locally symmetric space: $X\cong\Gamma\backslash G/K$ where $G$ is the Lorentz group
$SO(1,d)$, $\Gamma\cong\pi_1(X)$ is a torsion-free lattice of $G$ and $K\cong SO(d)$ is a maximal compact subgroup of $G$. 
Without loss of generality
we assume that  $K=\{{\rm diag}(1,k)\,|\,k\in SO(d)\}$.
The quotient space $G/K$ is isomorphic to the hyperboloid model 
$$\mathbb{H}^d=\Big\{x=(x_0,\dots,x_d)\in\mathbb{R}^{d+1}\,\Big|\,x_0^2-\sum_{i\ne0}x_i^2=1,\,x_0>0\Big\}.$$
Let $Y$ be a compact geodesic cycle of $X$ with dimension $n$ and hyperbolic measure $dy$. Then, up to finite cover,
$Y$ is isomorphic to $\Gamma_0
\backslash G_0/K_0$
where
$$G_0=\left\{\textup{diag}(\tau_1,\,\tau_2)\in G\,|\,\tau_1\in O(1,\,n),\,\,\tau_2\in O(d-n)\right\},$$
$$K_0=\left\{\textup{diag}(\rho_1,\,\rho_2)\in K\,|\,\rho_1\in
O(n),\,\,\rho_2\in O(d-n)\right\}$$ is the maximal compact subgroup
of $G_0$, and $\Gamma_0\cong\pi_1(Y)$ is a torsion-free
uniform lattice of $G_0$.

An automorphic representation $\pi$ of $G$ is called 
$G_0$-distinguished if there exists $\phi\in\pi$
such that $\int_{\Gamma_0\backslash
G_0}\phi(x)dx\ne 0$. 
The notion
of ``distinguished representation" is used in the setting of groups over adeles. Nevertheless, we adopt this notion since this paper applies to (uniform) real arithmetic lattices. 
In this 
prospect, Theorem \ref{thm} can be 
rephrased as follows: there are infinitely many (spherical) automorphic representations which are 
$G_0$-distinguished, provided that $\Gamma$ is uniform.

\subsection{Representation theory}\label{sym}
Let $\Theta$ be the Cartan involution on $G$: $\Theta(g)=\left(g^{T}\right)^{-1}$ (transpose inverse). The Cartan involution $\theta$ on Lie algebra level (i.e., $\theta(X)=-X^T$) gives rise to the vector space decomposition of the Lie algebra $\mathfrak{g}$ of $G$: $\mathfrak{g}=
\mathfrak{p}\oplus\mathfrak{k}$ where 
$\mathfrak{k}=
\{X\in\mathfrak{g}\,|\,\theta(X)=X\}$ is a  sub-Lie algebra of $\mathfrak{g}$
and $\mathfrak{p}=\{X\in\mathfrak{g}\,|\,\theta(X)=-X\}$.  Let $\mathfrak{a}$ be a maximal abelian subspace of $\mathfrak{p}$. For each linear functional $\alpha$ on $\mathfrak{a}$, define $\mathfrak{g}_{\alpha}=\left\{X\in\mathfrak{g}\,|\,[H,X]=\alpha(H)X~\textup{for all}~H\in\mathfrak{a}\right\}$.
The set of those nonzero $\alpha$ such that $\mathfrak{g}_{\alpha}\ne0$ is a root system, denoted as $(\mathfrak{g},\mathfrak{a})$. Let $E_{ij}=(e_{ij})$ be a $(d+1)\times(d+1)$ matrix whose entries satisfy: $e_{lk}=1$ for $(l,k)=(i,j)$, $e_{ik}=0$ otherwise. We choose $\mathfrak{a}=
\mathbb{R}\,E$
where $E=E_{12}+E_{21}$. Then the root system $(\mathfrak{g},\mathfrak{a})$ consists of two elements $\pm\alpha_0$ where $\alpha_0$ (the positive root) is defined by ${\rm ad}(E)$. Let $E_i=E_{1,2+i}+E_{2,2+i}+E_{2+i,1}-E_{2+i,2}$ ($1\leqslant i\leqslant d-1$), then ${\rm ad}(E)E_i=E_i$ for each $i$.
Hence,  $\mathfrak{g}_{\alpha_0}=\mathfrak{n}:=\mathbb{R}E_1\oplus\cdots\oplus\mathbb{R}E_{d-1}$. Define the groups $A$, $N$, $A^+$ to be $A=\exp(\mathfrak{a})$, $N=\exp(\mathfrak{n})$,
$A^+=\exp\left(
\mathbb{R}_+\,E\right)\subset A$ where $\mathbb{R}_+$ means the set of positive real numbers. We have the Iwasawa decomposition $G=NAK$ and $KAK$-decomposition $G=KA^+K$ both of which are unique (note that we have restricted the middle component $A(g)$ of $g$ in the $KAK$-decomposition to lie in $A^+$). 

Denote  
$$a_x^+=\exp(xE),\qquad n_{u}=\exp\left(\sum\limits_{i=1}^{d-1}u_iE_i\right)$$
for $x\in\mathbb{R}$,
$u=(u_1,\cdots,u_{d-1})\in\mathbb{R}^{d-1}$. It is easy to verify that
$$a_x^+=\begin{pmatrix}\cosh x&\sinh x&0&0&\cdots&0\\ \sinh x&\cosh x&0&0&\cdots&0\\0&0&1&0&\cdots&0\\
\vdots&\vdots&\vdots&\vdots&\vdots&\vdots\\
0&0&0&0&\cdots&1\end{pmatrix}$$ and
$$n_u=
\begin{pmatrix}
1+\frac{|u|^2}{2}&-\frac{|u|^2}{2}&u_1&u_2&\cdots&u_{d-1}\\
\frac{|u|^2}{2}&1-\frac{|u|^2}{2}&u_1&u_2&\cdots&u_{d-1}\\u_1&-u_1&1&0&\cdots&0\\u_2&-u_2&0&1&\cdots&0\\\vdots&\vdots&\vdots&\vdots&\vdots&\vdots\\
u_{d-1}&-u_{d-1}&0&0&\cdots&1\end{pmatrix}$$ where
$|u|^2=\sum_{i=1}^{d-1}u_i^2$.

The Killing form $B(X,Y)={\rm Tr}\big({\rm ad}(X){\rm ad}(Y)\big)$ on $\mathfrak{g}$, when restricted to $\mathfrak{p}$, induces a $G$-invariant Riemannian metric on $G/K$ with which we have the distance between the two points $g{\cdot}o$, $e{\cdot}o$ on $G/K$:
$${\rm d}_{G/K}(g{\cdot}o,e{\cdot}o)=
B\big(\log\,A(g),\log\,A(g)\big)^{1/2}=:\|g\|.$$
This invariant Riemannian metric on $G/K$ induces the metric and measure (denoted as $\mu^{\prime}$) on $\Gamma\backslash G/K$.
Let $dk$ be a Haar measure of $K$. Throughout the paper we always assume that ${\rm vol}(K)=\int_Kdk=1$. Any Haar measure of $G$ projects to a Radon measure $\mu$ of $\Gamma\backslash G$ and the latter, up to a positive scalar, projects to the measure $\mu^{\prime}$ on $\Gamma\backslash G/K$ such that the quotient integral formula holds: $\int_{\Gamma\backslash G}f(x)\mu(x)=\int_{\Gamma\backslash G/K}f(xk)\mu^{\prime}(x)dk$ for any $f\in C_c(\Gamma\backslash G)$.  

The group $G$ acts on  $L^2(\Gamma\backslash G,\mu)$ via the right regular translation $R_1$: $R_1(f)(x)=f(xg)$ for $f\in L^2(\Gamma\backslash G,\mu)$, $x\in\Gamma\backslash G$. The
Casimir operator $\square$ acts on the dense subset of smooth functions of $L^2(\Gamma\backslash G,\,\mu)$ as a
symmetric operator, and it has a unique self-adjoint extension to $L^2(\Gamma\backslash G,\,\mu)$; the similar conclusion holds for $\Delta$ and $L^2(\Gamma\backslash
G/K,\,\mu^{\prime})$ (see \cite{ch}). Any element in $L^2(\Gamma\backslash G/K,\mu^{\prime})$ can be viewed as an element in $L^2(\Gamma\backslash G,\mu)$ that is $K$-invariant under the action $R_1$. When restricted to $L^2(\Gamma\backslash G)^K$, the two operators $\square$ and $\Delta$ are identical with each other. 

Assume that $\Gamma$ is uniform. Then $R_1$ is decomposed into irreducible representations (see Theorem 9.2.2 of \cite{de}):
\begin{equation}\label{dec}
R_1\cong\bigoplus_{\pi\in\widehat{G}}N_{\Gamma}(\pi)\,\pi
\end{equation}
where $\widehat{G}$ denotes the unitary dual of $G$, i.e., the set
of equivalent classes of unitary irreducible representations of $G$,
$N_{\Gamma}(\pi)<\infty$ denotes the multiplicity of $\pi$.
Hence
\begin{eqnarray}\label{Kfix}
L^2(\Gamma\backslash G/K)\cong\bigoplus_{\pi
\in\widehat{G}^{\,K}}N_{\Gamma}(\pi)\,V_{\pi}^K
\end{eqnarray}
where $\widehat{G}^{\,K}$ means the subset of $\widehat{G}$ whose
element $\pi$ satisfies the condition $V_{\pi}^{K}\ne \{0\}$. Such
$\pi$'s are called \textit{spherical} or {\it class one} representations. Here
we use $V_{\pi}$ to denote the representation space of $\pi$.
Let $M\subset K$ denote the centralizer of $A$ in $K$. Then $M=\{{\rm diag}(1,1,k)\,|\,k\in SO(d-1)\}$. As $\mathfrak{a}$ is of dimension one, we can identify 
$\mathfrak{a}^{\ast}_{\mathbb{C}}$
with $\mathbb{C}$ via the map $\iota:\,\mathfrak{a}^{\ast}_{\mathbb{C}}
\rightarrow\mathbb{C}$,
$\alpha\mapsto\frac{d-1}{2}\alpha(E)$.
 Let $\rho$ be the half sum of positive
roots of the root system $(\mathfrak{g},\mathfrak{a})$, then $\iota(\rho)=\frac{d-1}{2}$. 
From now on we shall not distinguish $\mathfrak{a}^{\ast}_{\mathbb{C}}$
and $\mathbb{C}$. 
It is known that any nontrivial irreducible spherical representation of $G$ is equivalent to $I(\nu)={\rm Ind}_{MAN}^G(\bm{1}\otimes e^{\nu}\otimes\bm{1})$ for some $\nu\in(-\rho,\rho)\cup\,i\,\mathbb{R}$, and $I(\nu)\cong I(-\nu)$ for such $\nu$.  The trivial representation is isomorphic to the Langlands quotient of $I(\rho)$ modulo its unique subrepresentation.  Let 
$\nu\in\mathfrak{a}^{\ast}_{\mathbb{C}}$
and $(\sigma,V_{\sigma})$ be a representation of $M$.
Recall that 
$${\rm Ind}_{MAN}^G(\sigma\otimes e^{\nu}\otimes\bm{1})=\left\{h:\,G\rightarrow V_{\sigma}\,\Big|\,{h(mang)=e^{(\nu+\rho)\log a} \sigma(m)h(g)\,\,\textup{for}\atop \,\,man\in MAN,\,\,g\in G;\,\,h|_K\in L^2(K,\,V_{\sigma})}\right\}$$
endowed with the action $R_2$ of $G$: 
$$
R_2(g)h(x)=h\left(xg\right).$$
Let $\{\phi_j\}$ be 
an orthonormal basis of 
$L^2(\Gamma\backslash G/K,\mu^{\prime})$ 
such that each $\phi_j$ is a Maass form with Laplace eigenvalue  
$\lambda_j=\rho^2-\nu_j^2$. Denote by $\tilde{\phi}_i\in L^2(\Gamma\backslash G)$
the natural lift of $\phi_j$ such that $\tilde{\phi}_i(xk)=\phi_j(x{\cdot}o)$ for any $x\in\Gamma\backslash G$, $k\in K$. Under the action $R_1$ of $G$, $\tilde{\phi}_i$ generates an irreducible unitary subrepresentation $V_{\lambda_j}
\subset L^2(\Gamma\backslash G)$ of $G$. We have $V_{\lambda_j}\cong I(\nu_j)$. 

Let $G_0=M_0A_0N_0K_0$ be the  Langlands decomposition of the group $G_0$ where
$M_0=M\cap G_0$, $A_0=A\cap G_0=A$,
$N_0=N\cap G_0$.
Let $\mathfrak{n}_0$ be the Lie algebra of $N_0$.
The half sum of positive roots of the system $(\mathfrak{g}_0,\mathfrak{a}_0)$ is  $\rho_0=\frac{n-1}{2}$. 
Identifying $\mathfrak{p}_0=\mathfrak{a}\oplus
\mathfrak{n}_0$ with the tangent space of $G_0/K_0$ at $e{\cdot}o$, any geodesic on $G_0/K_0$
can be translated by certain $g\in G$ (via the left multiplication) to be a new geodesic which passes 
through $e{\cdot}o$ and 
is written as $\left\{\exp(tX)\,|\,t\in\mathbb{R}\right\}$ with proper $X\in
\mathfrak{p}_0$ (direction of the geodesic). By Proposition 5.13
of \cite{kn}:
\begin{equation*}
\mathfrak{p}_0=\bigcup\limits_{k\in K_0}\text{Ad}(k)\mathfrak{a},\end{equation*}
there exists $k\in K_0$ such that 
${\rm Ad}(k)X\in\mathfrak{a}$. Taking proper conjugation if necessary, we may assume that  
there exists a closed geodesic $C_0$ on 
$\Gamma_0\backslash G_0/K_0$ which can be written as $C_0=\left\{\exp(tX){\cdot}o\,|\,t\in[0,1]\right\}$ with some $X\in
\mathfrak{a}$, or equivalently,  $C_0=\Gamma_{00}\backslash A{\cdot}o$ where $\Gamma_{00}=\Gamma\cap AM_0=\Gamma\cap AM$ (the second identity holds since $\Gamma$ is torsion-free).

\subsection{Hyperbolic distance}\label{hd}
By Iwasawa decomposition, the 
subgroup $NA\subset G$ is topologically isomorphic to $\mathbb{H}^d\cong G/K$. We realize this isomorphisim as 
$$S:\,N\times A\rightarrow\mathbb{H}^d,\quad (n,\,a)\mapsto S(na)=na\cdot\xi_0$$
where $\xi_0=(1,0,\cdots,0)\in\mathbb{H}^d$. For $r>0$, define $a_r:=a_{\log r}^+=\exp(\log r\,E)$.
There is a one-to-one correspondence between $\mathbb{R}^{d-1}\times
\mathbb{R}_+$ and $NA$:$$T:\,\mathbb{R}^{d-1}\times
\mathbb{R}_+\rightarrow NA,\quad (u,r)\mapsto n_ua_r.$$
Let $x=(u,\,r)$,
$y=(v,\,t)\in\mathbb{R}^{d-1}\times
\mathbb{R}_+$. 
The hyperbolic distance 
between two points $a=S\circ T(x)$, $b=S\circ T(y)$
on $\mathbb{H}^d$ is 
\begin{equation}\label{eq-c}\textup{d}_{\mathbb{H}^d}(a,\,b)=\textup{arccosh}^+\left[\frac{|u-v|^2+r^2+t^2}{2rt}\right].\end{equation}
Here we use $^+$ to denote the nonnegative branch of the double valued function $\textup{arccosh}$. The relation between ${\rm d}_{\mathbb{H}^d}(\cdot,\,\cdot)$ and 
${\rm d}_{G/K}(\cdot,\,\cdot)$ is given by 
\begin{equation}\label{eq-d}
{\rm d}_{\mathbb{H}^d}(a,b)=
{\rm d}_{G/K}\big(T(x){\cdot}o,T(y){\cdot}o\big)=\left\|T(y)^{-1}\cdot T(x)\right\|.
\end{equation}
For this fact, see Proposition I.7.3 and I.7.5 of \cite{fj}.

\subsection{Trace formula}\label{tr}
Let $U$ be a subset of $G$,  $f$ be a
continuous function on $G$. Define
$$f_U:\,G\rightarrow\mathbb{R}_{\geqslant 0},\quad g\mapsto\sup\limits_{x,\,y\in U}\big|f(xgy)\big|.$$
We say $f$ is \textit{uniformly integrable}  if there exists some compact neighborhood $U$ of the unity $e$
such that $f_U$ lies in $L^1(G)$. Denote by $C_{\textup{unif}}(G)$ the
set of all continuous uniformly integrable functions over $G$.
Given $f\in C_{\textup{unif}}(G)$, define
$$(R_1(f)\phi)(x)=\int\limits_{G}f(g)R_1(g)\phi(x)dg$$
for $\phi\in L^2(\Gamma\backslash G)$. Then $R_1(f)$ is an integral
operator by
\begin{lem}\label{a}
$$(R_1(f)\phi)(x)=\int\limits_{\Gamma\backslash G}K_f(x,\,y)\phi(y)\mu(y),$$
where $K_f(x,\,y)=\sum_{\gamma\in\Gamma}f(x^{-1}\gamma y)$ is
continuous on $\Gamma\backslash G\times\Gamma\backslash G$.
\end{lem}
\noindent For details about $C_{\textup{unif}}(G)$ and the proof of this lemma, see Sect.\,9.2 of \cite{de}. The assumption in the reference, that $H$ is uniform, is necessary for the decomposition (\ref{dec}), but not for this lemma.

Let $f$ be a bi-$K$-invariant function in $C_{\textup{unif}}(G)$.
Then $R_1(f)$ acts on $V_{\lambda_j}^K\subset L^2(\Gamma\backslash G)^K$ with the
integral kernel $K_f$ since $R_1(f)\phi$ is still $K$-invariant for any $\phi\in V_{\lambda_j}^K$. The space $I(\nu_j)^K$ is one-dimensional:
any $K$-fixed function in $I(\nu_j)$ is determined by its values
at the points in $P=MAN$ thanks to the Langlands decomposition
$G=MANK$ and the the transformation law in $I(\nu_j)$. Consequently, $V_{\lambda_j}^K=\mathbb{C}\,\phi$ and there exists a scalar
$h_f(\nu_j)$ such that
$R_1(f)\phi=h_f(\nu_j)\phi$. In view of the bi-$K$-invariance of $f$ and the definition of $K_f$, we may identify $L^2(\Gamma\backslash G)^K$ with $L^2(\Gamma\backslash G/K)$, and regard $K_{f}(x,\,y)$ as
a function over $\Gamma\backslash G/K\times\Gamma\backslash G/K$. Then $R_1(f)$ induces the action of $f$ on $L^2(\Gamma\backslash G/K)$.
Likewise, $R_2(f)$ acts on $I(\nu_j)$:
$$(R_2(f)h)(x)=\int_Gf(g)R_2(g)h(x)dg,\quad h\in I(\nu_j)$$
with
integral kernel $K_f$ as above. Furthermore, we have  $R_2(f)\eta=h_f(\nu_j)\eta$ for any nontrivial
element $\eta$ in $I(\nu_j)^K$.
 
To compute $h_f(\nu_j)$, we use the model $I(\nu_j)$ and the action $R_2(f)$. Let $\eta_{\nu_j}\in I(\nu_j)^K$ be a complex-valued function over $G$ defined as $\eta_{\nu_j}(mank)=e^{(\nu_j+\rho)\log a}$. Since $\eta_{\nu_j}(1)=1$, it follows
that
$$(R_2(f)\eta_{\nu_j})(1)=h_f(\nu_j)\,\eta_{\nu_j}(1)=h_f(\nu_j).$$
By definition,\allowdisplaybreaks
\begin{eqnarray}\label{int}
(R_2(f)\eta_{\nu_j})(1)&=&\int\limits_{G}
f(g)\eta_{\nu_j}\left(g\right)dg\nonumber\\[0.2cm]
&\overset{(a)}{=}&\int\limits_{A}
\int\limits_{N}\int\limits_{K}
f\left(ank
\right)\eta_{\nu_j}(ank)dadndk\nonumber\\[0.2cm]
&=&
\int\limits_{N}\int\limits_{A}
f\left(an\right)
e^{(\nu_j+\rho)\log a}dadn
\end{eqnarray}
In the step (a) 
we have used the integral formula of functions on $G$ where the
variable is written in the $ANK$-order (see Corollary 5.3 of \cite{he}). Now we choose the Haar measures on $A$ and $N$.
Let $a=e^X$, $n=e^Y$ for $X\in\mathfrak{a}$, $Y\in\mathfrak{n}$.
Since $A$ and $N$ are abelian groups, $da:= dX$, $dn=dY$ are Haar measures on $A$, $N$ respectively, where
$dX$, $dY$ are Lebesgue measures on the Euclidean spaces $\mathfrak{a}$, $\mathfrak{n}$. Such choice of measures holds for any semisimple groups. The reason is that the group $N$ is nilpotent, while Lebesgue measures on its Lie algebras induce Haar measures of $N$ (see Theorem 2.1 of \cite{cg}). To be more precise, we have: $\int\limits_Nf(n)dn=\int\limits_{\mathfrak{n}}f(\exp Y)dY$ for any $f\in L^1(N,dn)$. Now (\ref{int}) reads
\begin{equation}\label{g-ker}
h_f(\nu_j)=
\int\limits_{\mathfrak{n}}
\int\limits_{\mathfrak{a}}
f\left(e^{X}\cdot
e^{Y}\right)e^{\nu_j(X)+\rho(X)}dXdY.
\end{equation}
We call $h_f(\nu_j)$ the {\it Harish-Chandra\,--\,Selberg transform} of $f$. The above formulation on $h_f(\nu_j)$ is due to Selberg \cite{se}. One can also use Harish-Chandra's 
theory on spherical functions to describe $h_f(\nu_j)$.

From now on we shall use $h_f(\lambda_j)$ instead of $h_f(\nu_j)$. This is reasonable: 
as $\nu_j$ is decided up to $\pm1$ for fixed $\lambda_j$, and $I(\nu_j)\cong I(-\nu_j)$ for $\nu_j\in(-\rho,\rho)\cup\,i\,\mathbb{R}$, so we have $h_f(\nu_j)=h_f(-\nu_j)$.
Assume that $f\in C_{\textup{unif}}(G)$ is a bi-$K$-invariant function
such that the series
$$k_f(z,\,w):=\sum\limits_{j=0}^{\infty}
h_f(\lambda_j)\phi_j(z)\overline{\phi_j(w)},\quad
z,\,w\in\Gamma\backslash G/K$$ locally uniformly converge everywhere. 
\begin{prop}\label{trace-kernel}$K_f$ being viewed as a function over $\Gamma\backslash G/K
\times\Gamma\backslash G/K$, we have: $ K_f=k_f$.
\end{prop}
\begin{proof} 
By Lemma \ref{a}, $R_1(f)$ is an integral operator with
continuous integral kernel $K_f$. Meanwhile
$R_1(f)\phi_j=h_f(\lambda_j)\phi_j$. Define
$$T_k:\,L^2(\Gamma\backslash G/K)\rightarrow L^2(\Gamma\backslash G/K),\quad \phi\mapsto\int\limits_{ \Gamma\backslash G/K}k_f(z,\,w)\phi(w)\mu^{\prime}(w).$$
Then, by definition $T_k$ is an integral operator such that
$T_k(\phi_j)=h_f(\lambda_j)\,\phi_j$ as $\phi_j$'s are
orthonormal to each other. Hence, $T_k$ and $R(f)$ are identical to each
other as operators and their integral kernels are equal to each
other except on a possible subset of measure zero. The locally uniform convergence of $k_f$ implies that $k_f$ is a
continuous function as all $\phi_j$'s are analytic over $\Gamma\backslash G/K$. It follows that $K_f=k_f$.
\end{proof}
By Lemma \ref{a} and Proposition \ref{trace-kernel} one has
$$\sum\limits_{\gamma\in\Gamma}f(z^{-1}\gamma
w)=\sum\limits_{j=0}^{\infty}
h_f(\lambda_j)\phi_j(z)
\overline{\phi_j(w)},\quad z,~w\in\Gamma\backslash G/K.$$
Let $\psi$ be a normalized Maass form on 
$\Gamma_0\backslash G_0/K_0$ with Laplace eigenvalues $\lambda
=\rho_0^2-\nu^2$ where $\nu\in(-\rho_0,\rho_0)\cup\,i\,\mathbb{R}$.
Integrating both sides on $\Gamma_0\backslash G_0/K_0\times \Gamma_0\backslash G_0/K_0$ with respect to the measure
$\psi(z)\overline{\psi(w)}
dzdw$ gives the (relative) trace formula
\begin{equation}\label{eq-b}
\int_{\Gamma_0\backslash G_0/K_0}\int_{\Gamma_0\backslash G_0/K_0}\sum\limits_{\gamma\in\Gamma}f(z^{-1}\gamma
w)\psi(z)\overline{\psi(w)}dzdw
=\sum\limits_{j=0}^{\infty}
h_f(\lambda_j)\,
P_Y(\phi_j,\psi)\,\overline{P_Y(\phi_j,\psi)}.
\end{equation}
The left (right) hand side of (\ref{eq-b}) is called {\it geometric} (resp. {\it spectral}\,) side. For this identity to hold, the {\it test function} $f$ should satisfy: (1) $f\in C_{\rm unif}(G)$; (2) $f$ is bi-$K$-invariant; (3) $k_f$ is locally uniformly convergent. These conditions will be checked in Sect.\,\ref{fkf} for a special $f$ chosen in the next section.

\section{The spectral side}\label{spec}
In this section we choose a test function and apply it to the spectral side of (\ref{eq-b}). 
The bi-$K$-invariance
of $f$ indicates that $f(g)$ depends exactly on the hyperbolic distance  between $g{\cdot}o$ and $e{\cdot}o$ (on $G/K$). Let $\Phi_{\mu}(x)$ be a smooth function over $\mathbb{R}_{\geqslant0}$ for any $\mu\in\mathbb{R}_+$. Define the test function $f\in C^{\infty}(G)$ as
$f(g)=\Phi_{\mu}\left({\rm d}_{G/K}(g{\cdot}o,\,e{\cdot}o)\right)$.  
Like the case of $G/K$, the quotient $G_0/K_0\cong\mathbb{H}^{n}$ can be parameterized by $\mathbb{R}^{n-1}
\times\mathbb{R}_+$ via the maps $T$, $S$ (see Sect.\,\ref{hd}).
By (\ref{g-ker}) we have
\begin{eqnarray*}
h_f(\lambda_j)&=&
\int\limits_{\mathbb{R}^{d-1}}
\int\limits_{\mathbb{R}}
\Phi_{\mu}\Big(
{\rm d}_{G/K}\left(a_x^+n_u{\cdot}o,\,e{\cdot}o\right)\Big)\cdot e^{(\nu_j+\rho)x}dxdu\\[0.2cm]
&=&\int\limits_{\mathbb{R}^{d-1}}
\int\limits_{\mathbb{R}_+}
\Phi_{\mu}\Big(
{\rm d}_{G/K}\left(a_rn_u{\cdot}o,\,e{\cdot}o\right)\Big)
\cdot r^{\nu_j+\rho-1}\,drdu
\end{eqnarray*}
where we have made the variable exchange $x\rightarrow\log r$ in the second step.
By (\ref{eq-c}) and (\ref{eq-d}) we have
 $$
{\rm d}_{G/K}\left(a_rn_u{\cdot}o,\,e{\cdot}o\right)={\rm
d}_{G/K}\left(e{\cdot}o,\,n_{-u}a_{r^{-1}}{\cdot}o\right)=
\text{arccosh}^+\left(\frac{|ru|^2+1+r^2}{2r}\right),
$$
noting that $e=n_{0}\,a_{1}$.  Originally
one would like to insert the heat kernel (see \cite{gn}), but then it is difficult to deal with the geometric
side. 
In this paper the test function is chosen to be  $$\Phi_{\mu}(x)=\exp(-\mu\cdot\cosh x).$$ It follows that
\begin{eqnarray}\label{eq-e}
h_f(\lambda_j)=
\int\limits_{\mathbb{R}^{d-1}}
\int\limits_{\mathbb{R}_+}
\exp\left[-\mu\left(\frac{|u|^2+1}{2}r+\frac{\frac{1}{2}}{r}\right)\right]
r^{\nu_j+\rho-1}drdu.
\end{eqnarray}
The
following two integral formulas on $K$-Bessel functions are useful to us:
\begin{equation}\label{eq-f}
\int\limits_{0}^{\infty}x^{\nu-1}\exp\left(-\frac{\alpha}{x}-\beta
x\right)dx=2\left(\frac{\alpha}{\beta}\right)^{\frac{\nu}{2}}K_{\nu}\left(2\sqrt{\alpha\beta}\right),\quad\text{Re}(\alpha)>0,\,\,
\text{Re}(\beta)>0.
\end{equation}
\begin{equation}\label{eq-g}
\int\limits_{0}^{\infty}
\left(x^2+b^2\right)^{\mp\frac{\nu}{2}}K_{\nu}\left(a\sqrt{x^2+b^2}\right)\cos(cx)dx=\sqrt{\frac{\pi}{2}}a^{\mp\nu}b^{\frac{1}{2}\mp\nu}\left(a^2+c^2\right)^{\pm\frac{\nu}{2}-\frac{1}{4}}K_{\pm\nu-\frac{1}{2}}\left(b\sqrt{a^2+c^2}\right)
\end{equation}
where $\text{Re}(a)>0$, $\text{Re}(b)>0$, $c$ is a real number.
These are the formulas 3.471.9 and 6.726.4 of \cite{gr}
respectively.

Let $\alpha=\frac{\mu}{2}$, $\beta=\frac{|u|^2+1}{2}\,\mu$,
$\nu=\nu_j+\rho$ in the formula (\ref{eq-f}),  then the integration along $r$ in (\ref{eq-e}) gives
\begin{equation*} 
h_f(\lambda_j)
=\int\limits_{\mathbb{R}^{d-1}}2\left(|u|^2+1\right)^{-\frac{\nu_j+\rho}{2}}K_{\nu_j+\rho}\left(\mu\sqrt{|u|^2+1}\right)du
\end{equation*}
Let $x=u_1$, $b^2=u_2^2+\cdots+u_{d-1}^2+1$, $a=\mu$, $c=0$,
$\nu=\nu_j+\rho$ in the (first case of) formula (\ref{eq-g}), then the integration along $u_1$ in the above integral is equal to 
\begin{equation}\label{eq-i}
2^d\int\limits_{0}^{\infty}\cdots\int\limits_{0}^{\infty}
\sqrt{\frac{\pi}{2\mu}}\left(\sqrt{u_2^2+\cdots+u_{d-1}^2+1}\right)^{\frac{1}{2}-(\nu_j+\rho)} 
K_{\nu_j+\rho-\frac{1}{2}}\left(\mu\sqrt{u_2^2+\cdots+u_{d-1}^2+1}\right)du_2\cdots
du_{d-1}
\end{equation}
Let $x=u_2$, $b^2=u_3^2+\cdots+u_{d-1}^2+1$, $a=\mu$, $c=0$,
$\nu=\nu_j+\rho-\frac{1}{2}$ in the formula (\ref{eq-g}), then the integration along $u_2$ in the above integral gives
\begin{equation}
(\ref{eq-i})=2^d\left(\sqrt{\frac{\pi}{2\mu}}\,\right)^2
\int\limits_{0}^{\infty}\cdots\int\limits_{0}^{\infty}
\left(\sqrt{u_3^2+\cdots+u_{d-1}^2+1}\right)^{1-(\nu_j+\rho)} 
K_{\nu_j+\rho-1}\left(\mu\sqrt{u_3^2+\cdots+u_{d-1}^2+1}\right)du_3\cdots
du_{d-1}\notag
\end{equation}
Repeating this process, i.e., doing integrations along $u_3$,
$u_4$, \dots, $u_{d-1}$ step by step in the above fashion, we
finally get\index{$h_f(\lambda_j)$}
\begin{equation*}
h_f(\lambda_j)=2^d\left(\sqrt{\frac{\pi}{2\mu}}\,\right)^{d-1} K_{\nu_j+\rho-\frac{d-1}{2}}(\mu)=2^d\left(\sqrt{\frac{\pi}{2\mu}}\,\right)^{d-1}K_{\nu_j}(\mu).
\end{equation*}
Now the spectral side of (\ref{eq-b}) reads:
$$\sum\limits_{j=0}^{\infty}2^d\left(\sqrt{\frac{\pi}{2\mu}}\,\right)^{d-1}\, K_{\nu_j}(\mu)\,\big|P_Y(\phi_j,\psi)\big|^2.$$

\section{The geometric side}\label{geo}
Under our choice of $f$ 
the geometric side of (\ref{eq-b}) splits as follows.
\begin{eqnarray*}
&&\int_{\Gamma_0\backslash G_0/K_0}\int_{\Gamma_0\backslash G_0/K_0}\sum\limits_{\gamma\in\Gamma}f(z^{-1}\gamma
w)\psi(z)\overline{\psi(w)}dzdw\notag\\[0.2cm]
&=&
\int_{\Gamma_0\backslash G_0/K_0}
\int_{\Gamma_0\backslash G_0/K_0}\,
\sum\limits_{\gamma\in
\Gamma_0}\Phi_{\mu}\left({\rm d}_{G/K}
(\gamma w,
\,z)\right)\psi(z)
\overline{\psi(w)}dzdw
\notag\\[0.2cm]
&&\quad+\,\mathcal{O}\left(
\int_{\Gamma_0\backslash G_0/K_0}
\int_{\Gamma_0\backslash G_0/K_0}\,\sum\limits_{\gamma_1,\,\gamma_2\in\Gamma_0}\,
\sum\limits_{\tilde{\gamma}
\in\Gamma_0\backslash
\Gamma/\Gamma_0
\smallsetminus
\{\tilde{1}\}}
\Phi_{\mu}\left({\rm d}_{G/K}
\left(\gamma\gamma_1w,
\,\gamma_2z\right)\right)
\psi(z)
\overline{\psi(w)}dzdw\right)
\notag\\[0.2cm]
&=&\int_{G_0/K_0}
\int_{\Gamma_0\backslash G_0/K_0}\,
\Phi_{\mu}\left({\rm d}_{G/K}(w,z)\right)\psi(z)
\overline{\psi(w)}dzdw
\notag\\[0.2cm]
&&\quad+\,\mathcal{O}\left(
\int_{G_0/K_0}
\int_{G_0/K_0}\,\,
\sum\limits_{\tilde{
\gamma}
\in\Gamma_0\backslash
\Gamma/\Gamma_0
\smallsetminus\{
\tilde{1}\}}
\Phi_{\mu}\left({\rm d}_{G/K}(\gamma w,z)\right)\psi(z)
\overline{\psi(w)}dzdw\right)
\end{eqnarray*}
where $\tilde{\gamma}$ denotes a nontrivial double coset  in 
$\Gamma_0\backslash
\Gamma/\Gamma_0$. To simplify the notations, we still use $dz$, $dw$ to denote the measure  of space $G_0/K_0$. This is reasonable since the measure of $G_0/K_0$ descends to that of the quotient space $\Gamma_0\backslash G_0/K_0$. 
It is not clear whether any 
element in 
$\Gamma\smallsetminus
\Gamma_0$
can be written as $\gamma_1
\gamma\gamma_2$
with unique $\gamma_1$, $\gamma_2\in\Gamma_0$ and some fixed representative element $\gamma$ of a 
double coset class $\tilde{\gamma}
\ne\tilde{1}$. So we have to use the expression $\mathcal{O}(\,\cdots)$ in the above formula.
Denote
$$\Sigma_0=\int_{G_0/K_0}
\int_{\Gamma_0\backslash G_0/K_0}\,
\Phi_{\mu}\left({\rm d}_{G/K}(w,z)\right)\psi(z)
\overline{\psi(w)}dzdw$$
and
$$\Sigma_1=\int_{G_0/K_0}
\int_{G_0/K_0}\,
\sum\limits_{\tilde{\gamma}
\in\Gamma_0\backslash
\Gamma/\Gamma_0
\smallsetminus\left\{\tilde{1}\right\}}
\Phi_{\mu}\left({\rm d}_{G/K}(\tilde{\gamma}w,z)\right)\psi(z)
\overline{\psi(w)}dzdw.$$
In the next two sections we shall show that $\Sigma_0$ is the main term of the geometric side, while $\Sigma_1$ is the error term.  

\subsection{The main term}\label{mt}
Like the case of $\phi\in L^2(\Gamma\backslash G/K)$, the Maass form $\psi$ also gives rise to an irreducible unitary spherical representation $V^{\prime}_{\lambda}\subset L^2(\Gamma_0\backslash G_0)$. Assume that $V^{\prime}_{\lambda}\cong I^{\prime}(\nu):={\rm Ind}_{M_0AN_0}^{G_0}(\bm{1}\otimes e^{\nu}\otimes\bm{1})$.
Here, to distinguish the representation of $G_0$ from that of $G$, we use $V^{\prime}_{\lambda}$ 
and $I^{\prime}(\nu)$, instead of $V_{\lambda}$ and
$I(\nu)$. 
Normalize the Haar measure of  $K_0$ such that ${\rm vol}(K_0)=1$. By the quotient integral formula, we can rewrite $\Sigma_0$ as
$$\Sigma_0=\int_{G_0}
\int_{\Gamma_0\backslash G_0}\,
\Phi_{\mu}\left({\rm d}_{G/K}(g{\cdot}o,x{\cdot}o)\right)\tilde{\psi}(x)
\overline{\tilde{\psi}(\bar{g})}dxdg$$
where $\tilde{\psi}$, as before, stands for the natural lift of $\psi$ on $\Gamma_0\backslash G_0$, $\bar{g}$ denotes the element $\Gamma_0{\cdot} g\in\Gamma_0\backslash G_0$, $dx$ the invariant Radon measure of 
$\Gamma_0\backslash G_0$,
$dg$ the Haar measure of $G_0$. In such  formulation, $\Sigma_0$ defines a nonzero $(K_0\times K_0)$-invariant functional $L^{\rm aut}_{\nu}$ over $V_{\lambda}^{\prime}\times
V_{\lambda}^{\prime}$
where each $K_0\times K_0$ acts on $V_{\lambda}^{\prime}\times
V_{\lambda}^{\prime}$
via $R_1\times R_1$ (the right regular translation):
$$L^{\rm aut}_{\nu}:\,V_{\lambda}^{\prime}\times
V_{\lambda}^{\prime}
\rightarrow\mathbb{C},\quad
(h_1,h_2)\mapsto\int_{G_0}
\int_{\Gamma_0\backslash G_0}\Phi_{\mu}\left({\rm d}_{G/K}(g{\cdot}o,x{\cdot}o)\right)
h_1(x)\overline{h_2(\bar{g})}
dxdg.
$$ This functional is $\mathbb{C}$-linear for its first 
entry and conjugate $\mathbb{C}$-linear for its second entry.
For $z\in\Gamma_0\backslash G_0$, $L^{\rm aut}_{\nu}$ is well-defined although ${\rm d}_{G/K}
\left(w{\cdot}o,
\,z{\cdot}o\right)$ is not. 
The space of $(K_0\times K_0)$-invariant functionals over $V_{\lambda}^{\prime}\times
V_{\lambda}^{\prime}$
is one-dimensional. The reason is as follows. In view of the equivalence between $V_{\lambda}^{\prime}$ and $I^{\prime}(\nu)$, it suffices to show that the space of $K_0$-invariant functional over $I^{\prime}(\nu)$ is one dimensional, which is clearly true by the definition of $I^{\prime}(\nu)$. As a consequence, there exists a nonzero scalar $a_{\nu}\in\mathbb{C}$ such that $$L^{\rm aut}_{\nu}(h_1,h_2)=a_{\nu}\cdot L^{\rm mod}_{\nu}(f_1,f_2)$$ for any nonzero $(K_0\times K_0)$-invariant $(\mathbb{C}\times
\overline{\mathbb{C}})$-linear functional $L^{\rm mod}_{\nu}$ on 
$I^{\prime}(\nu)
\times I^{\prime}(\nu)$ where $f_i\in I^{\prime}(\nu)$ corresponds to $h_i\in
V_{\lambda}^{\prime}$.
Define $L^{\rm mod}_{\nu}$ to be 
$$L^{\rm mod}_{\nu}(f_1,f_2)
=\int_{G_0}
\int_{\Gamma_0\backslash G_0}
\Phi_{\mu}\left({\rm d}_{G/K}(g{\cdot}o,x{\cdot}o)\right)
f_1(x)\overline{f_2(g)}dzdg,\quad f_1,~f_2\in
I^{\prime}(\nu).
$$
Then
$\Sigma_0=a_{\nu}\cdot L^{\rm mod}_{\nu}(\eta_{\nu},\eta_{\nu})$.
See Sect.\,\ref{tr} for the definition of $\eta_{\nu}$.
Now we compute $L^{\rm mod}_{\nu}(\eta_{\nu},\eta_{\nu})$. 
For $g{\cdot}o=n_ua_s{\cdot}o\in G_0/K_0$ where $u\in\mathbb{R}^{n-1}$, $s\in\mathbb{R}_+$, equip  $G_0/K_0$ with the hyperbolic measure
$d(g{\cdot}o)=\frac{dsdu}{s^n}$. Let
$\mathcal{P}(Y)$ be the completion of a subset
of $\mathbb{R}^{n-1}\times
\mathbb{R}_+$ that is isomorphic to  
$\Gamma_0\backslash G_0/K_0$. Then $\mathcal{P}(Y)$ is compact. The following commutativity property will be used frequently in this paper:
\begin{equation}\label{commu}
a_rn_u=n_{ur}a_r.
\end{equation} 
One can verify (\ref{commu}) be a direct computation, or see Proposition I.4.2 of \cite{fj}.
Actually (\ref{commu}) results from a simple fact in Lie algebra: 
${\rm ad}(E)E_i=E_i$.
As $\eta_{\nu}$ is $K_0$-invariant and ${\rm d}_{G/K}(w,z)={\rm d}_{G_0/K_0}(w,z)$ for $w$, $z\in G_0/K_0$ (note that $G_0/K_0\subset G/K$ is totally geodesic), we have:
\begin{eqnarray*}
L^{\rm mod}_{\nu}(\eta_{\nu},\eta_{\nu})
&=&\int_{G_0/K_0}
\int_{\Gamma_0\backslash G_0/K_0}
\Phi_{\mu}\left({\rm d}_{G_0/K_0}(w,z)\right)
\eta_{\nu}(z)
\overline{\eta_{\nu}(w)}dzdw
\notag\\[0.2cm]
&=&\int\limits_{\mathbb{R}^{n-1}
\times\mathbb{R}_+}\,
\int\limits_{\mathcal{P}(Y)}
\Phi_{\mu}\left({\rm d}_{G_0/K_0}(n_ua_s{\cdot}o,
\,n_va_r{\cdot}o)\right)
r^{\nu+\rho_0}
s^{\bar{\nu}+\rho_0}
\frac{drdv}{r^n}
\frac{dsdu}{s^n}\notag
\\[0.2cm]
&=&
\int\limits_{\mathcal{P}(Y)}\,
\int\limits_{\mathbb{R}^{n-1}
\times\mathbb{R}_+}
\exp\left(
-\mu\cdot\frac{|u-v|^2+s^2+r^2}{2sr}
\right)
s^{\bar{\nu}+\rho_0}
\frac{dsdu}{s^n}\cdot
r^{\nu+\rho_0}
\frac{drdv}{r^n}\notag
\\[0.2cm]
&=&
\int\limits_{\mathcal{P}(Y)}\,
\int\limits_{\mathbb{R}^{n-1}
\times\mathbb{R}_+}
\exp\left[
-\frac{\mu}{2}\left(\frac{\left|\frac{u-v}{r}\right|^2+1}{\frac{s}{r}}+\frac{s}{r}
\right)\right]
s^{\bar{\nu}+\rho_0}
\frac{dsdu}{s^n}
\cdot
r^{\nu+\rho_0}
\frac{drdv}{r^n}
\end{eqnarray*}
Let $u^{\prime}=\frac{u-v}{r}$, $s^{\prime}=\frac{s}{r}$, then $du=r^{n-1}du^{\prime}$ and
\begin{equation*}
L^{\rm mod}_{\nu}(\eta_{\nu},\eta_{\nu})
=\int\limits_{\mathcal{P}(Y)}\,
\int\limits_{\mathbb{R}^{n-1}
\times\mathbb{R}_+}
\exp\left[
-\frac{\mu}{2}\left(\frac{\left|u^{\prime}\right|^2+1}{s^{\prime}}+s^{\prime}
\right)\right]
s^{\prime\,(\bar{\nu}+\rho_0-n)}ds^{\prime}du^{\prime}
\cdot
r^{\bar{\nu}+\nu+2\rho_0-n} 
drdv.
\end{equation*}
The rest of the computation is merely a copy of that for $h_f(\lambda_j)$. Firstly,
apply (\ref{eq-f}) to the integration over $s^{\prime}$, then the right hand side of the above identity is equal to
\begin{equation*}
2\int\limits_{\mathcal{P}(Y)}\,
\int\limits_{\mathbb{R}^{n-1}}
\left(1+|u^{\prime}|^2\right)^{\frac{\bar{\nu}+\rho_0-n+1}{2}}K_{\bar{\nu}+\rho_0-n+1}\left(\mu\sqrt{1+|u^{\prime}|^2}\right)
du^{\prime}\cdot
r^{2 {\rm Re}(\nu)-1}
drdv.
\end{equation*}
Secondly, apply the second case of (\ref{eq-g}) to 
the integration over $u^{\prime}$ step by step, then the above integral is equal to   
$$2^n\left(\sqrt{\frac{\pi}{2\mu}}\,\right)^{n-1}K_{\bar{\nu}}(\mu)
\int_{\mathcal{P}(Y)}r^{2 {\rm Re}(\nu)-1}
drdv.$$
The subset $\mathcal{P}(Y)$ is compact and $r>0$ for $(v,r)\in\mathcal{P}(Y)$. Hence, the  integral $$I_{\nu}:=\int_{\mathcal{P}(Y)}r^{2 {\rm Re}(\nu)-1}
drdv$$ converges and does not vanish. Denote $b_{\nu}=2^na_{\nu}\, I_{\nu}\ne 0$. Up to now we have shown 
\begin{equation}\label{s1}
\Sigma_0=b_{\nu}\cdot
\left(\sqrt{\frac{\pi}{2\mu}}\,\right)^{n-1}K_{\bar{\nu}}(\mu).
\end{equation}

\subsection{The error term}\label{et}
In this section we give a bound for $\Sigma_1$. The main conclusion is
\begin{equation}\label{error}
\Sigma_1\ll\mu^{-(n+2)/2}e^{-\mu}.
\end{equation}
It turns out that the error term 
$\Sigma_1$ is more difficult to be treated than $\Sigma_0$. 
Like what we have done for 
$\Sigma_0$, we use the uniqueness of  $(K_0\times K_0)$-invariant $(\mathbb{C}\times
\overline{\mathbb{C}})$-linear functionals to reduce the computation of $\Sigma_1$ to
that of special integrals. 
Define 
$$\ell^{\rm aut}_{\nu}:\,V^{\prime}_{\lambda}\times V^{\prime}_{\lambda}
\rightarrow\mathbb{C},\quad(h_1,h_2)\mapsto
\int_{G_0}\int_{G_0}
\sum\limits_{\tilde{\gamma}\in
\Gamma_0\backslash\Gamma/
\Gamma_0\smallsetminus
\left\{\tilde{1}\right\}}
\Phi_{\mu}\big({\rm d}_{G/K}(\gamma g_1{\cdot}o,g_2{\cdot}o)\big)
h_1(\bar{g}_1)
\overline{h_2(\bar{g}_2)}dg_1dg_2$$
where $\bar{g}_i$ denotes the element $\Gamma_0\cdot g_i\in
\Gamma_0\backslash G_0$. It is clear that $\ell^{\rm aut}_{\nu}$ is a nonzero $(K_0\times K_0)$-invariant $(\mathbb{C}\times
\overline{\mathbb{C}})$-linear functional on $V^{\prime}_{\lambda}\times V^{\prime}_{\lambda}$ and 
$\Sigma_1=\ell^{\rm aut}_{\nu}
\left(\tilde{\psi},\tilde{\psi}\,\right)$. As before, the space of such functionals on is one-dimensional. Thus, for a given nonzero $(K_0\times K_0)$-invariant $(\mathbb{C}\times
\overline{\mathbb{C}})$-linear functional on $I^{\prime}(\nu)\times I^{\prime}(\nu)$ there exists a scalar $d_{\nu}\in\mathbb{C}$ such that
$\ell^{\rm aut}_{\nu}=d_{\nu}\cdot \ell^{\rm mod}_{\nu}$. 
Note that $d_{\nu}$ depends only on $\nu$.
 Define
$$\ell^{\rm mod}_{\nu}:\,I^{\prime}(\nu)\times I^{\prime}(\nu)\rightarrow\mathbb{C},\quad
(f_1,f_2)\mapsto\int_{G_0}\int_{G_0}
\sum\limits_{\tilde{\gamma}\in
\Gamma_0\backslash\Gamma/
\Gamma_0\smallsetminus
\left\{\tilde{1}\right\}}
\Phi_{\mu}\left({\rm d}_{G/K}(g_2{\cdot}o,g_1{\cdot}o)\right)
f_1(g_1)\overline{f_2(g_2)}dg_1dg_2.$$
Then $\ell^{\rm mod}_{\nu}$ is 
$(K_0\times K_0)$-invariant
$(\mathbb{C}\times
\overline{\mathbb{C}})$-linear on $I^{\prime}(\nu)\times I^{\prime}(\nu)$ and 
$\Sigma_1=d_{\nu}\cdot\ell^{\rm mod}_{\nu}(\eta_{\nu},\eta_{\nu})$ where $\eta_{\nu}$ is as before.
Let $w=n_u\,a_r{\cdot}o$, $z=n_v\,a_t{\cdot}o\in G_0/K_0$ where $u$, $v\in\mathbb{R}^{n-1}$
and
$r$, $t\in\mathbb{R}_+$. 
We have
$$\ell^{\rm mod}_{\nu}(\eta_{\nu},\eta_{\nu})=
\int\limits_{
\mathbb{R}^{n-1}\times\mathbb{R}_+}
\,\int\limits_{\mathbb{R}^{n-1}
\times\mathbb{R}_+}
\sum\limits_{\tilde{\gamma}\in
\Gamma_0\backslash\Gamma/
\Gamma_0\smallsetminus
\left\{\tilde{1}\right\}}
\Phi_{\mu}\big({\rm d}_{G/K}(\gamma w,z)\big)
t^{\nu+\rho_0}r^{\bar{\nu}+\rho_0}
\frac{dtdv}{t^n}\frac{drdu}{r^n}.
$$ 
Write $\gamma=a(\gamma)n(\gamma)k(\gamma)=
a_{r_0}\,
n_{w_0}\,
\begin{pmatrix}
1&0\\0&u_0
\end{pmatrix}\in ANK$
where
$w_0=(w_{01},\cdots,w_{0,d-1})\in\mathbb{R}^{d-1}$
and $u_0=(u_{ij})\in SO_d$. 
Assume that $\gamma w=n_{v_1}a_{s_1}{\cdot}o$ where $v_1=(v_{11},\cdots, v_{1,d-1})$. Denote
\begin{eqnarray*}
J_{\gamma}&=&\int\limits_{
\mathbb{R}^{n-1}\times\mathbb{R}_+}
\,\int\limits_{\mathbb{R}^{n-1}
\times\mathbb{R}_+}
\Phi_{\mu}\big({\rm d}_{G/K}(\gamma w,z)\big)
t^{\nu+\rho_0}r^{\bar{\nu}+\rho_0}
\frac{dtdv}{t^n}\frac{drdu}{r^n}\\[0.2cm]
&=&\int\limits_{
\mathbb{R}^{n-1}\times\mathbb{R}_+}
\,\int\limits_{\mathbb{R}^{n-1}
\times\mathbb{R}_+}
\exp\left(-\mu\cdot\frac{|v-v_1|^2+s_1^2+t^2}{2s_1t}\right)
t^{\nu+\rho_0-n}r^{\bar{\nu}+\rho_0-n}
dtdvdrdu
\end{eqnarray*}
We treat the integration along $t$ by using (\ref{eq-f}) and get
\begin{eqnarray*}
J_{\gamma}&=&2\int\limits_{
\mathbb{R}^{n-1}\times\mathbb{R}_+}
\,\int\limits_{\mathbb{R}^{n-1}}
\left(|v-v_1|^2+s_1^2\right)^{\frac{\nu+\rho_0-n+1}{2}}K_{\nu+\rho_0-n+1}\left(\mu\sqrt{\left|
\frac{v-v_1}{s_1}\right|^2+1}\right)r^{\bar{\nu}+\rho_0-n}dvdrdu\\[0.2cm]
&=&
2\int\limits_{
\mathbb{R}^{n-1}\times\mathbb{R}_+}
\,\int\limits_{\mathbb{R}^{n-1}}
\left(\left|\frac{v-v_1}{s_1}\right|^2+1\right)^{\frac{\nu+\rho_0-n+1}{2}}
K_{\nu+\rho_0-n+1}\left(\mu\sqrt{\left|
\frac{v-v_1}{s_1}\right|^2+1}\right)d\left(\frac{v}{s_1}\right)s_1^{\nu+\rho_0}r^{\bar{\nu}+\rho_0-n}drdu
\end{eqnarray*}
The first $(n-1)$ entries in $\frac{v_1}{s_1}$ can be absorbed in to $\frac{v}{s_1}$ when we do the integration along $\frac{v}{s_1}$, leaving the last 
$(d-n)$ components of $\frac{v_1}{s_1}$. Denote $|x|^2_{\geqslant n}=x_n^2+\cdots+x_{d-1}^2$ for $x=(x_1,\cdots,x_{d-1})\in\mathbb{R}^{d-1}$. Let $v^{\prime}=\frac{v}{s_1}$, then 
$$
J_{\gamma}=2\int\limits_{
\mathbb{R}^{n-1}\times\mathbb{R}_+}
\,\int\limits_{\mathbb{R}^{n-1}}
\left(\left|v^{\prime}\right|^2+\left|\frac{v_1}{s_1}\right|_{\geqslant n}^2+1\right)^{\frac{\nu+\rho_0-n+1}{2}}
K_{\nu+\rho_0-n+1}\left(\mu\sqrt{\left|v^{\prime}
\right|^2+\left|
\frac{v_1}{s_1}\right|_{\geqslant n}^2+1}\right)dv^{\prime}
s_1^{\nu+\rho_0}
r^{\bar{\nu}+\rho_0-n}drdu
$$
Applying the second case of (\ref{eq-g}) to $dv^{\prime}$ step by step \big(just as what we have done for $h_f(\lambda_j)$ and $\Sigma_0$\big) gives
$$J_{\gamma}=2^n
\left(\sqrt{\frac{\pi}{2\mu}}\,\right)^{n-1}
\int\limits_{
\mathbb{R}^{n-1}\times\mathbb{R}_+}
\left(\left|\frac{v_1}{s_1}\right|_{\geqslant n}^2+1\right)^{\frac{\nu}{2}}
K_{\nu}\left(\mu\sqrt{\left|
\frac{v_1}{s_1}\right|_{\geqslant n}^2+1}\right)
s_1^{\nu+\rho_0}
r^{\bar{\nu}+\rho_0-n}drdu.$$
The computation shows \big(one should distinguish $w$ in below from the $w$ that has appeared earlier (as a point on $\Gamma_0\backslash G_0/K_0$ or $G_0/K_0$)\big):
$$n_{w}a_{s}k=\begin{pmatrix}
\frac{s+s^{-1}}{2}+\frac{s^{-1}}{2}\,|w|^2,~&\cdots\\[0.2cm]
\frac{s-s^{-1}}{2}+\frac{s^{-1}}{2}\,|w|^2,&\cdots\\[0.2cm]
w_{1}\cdot s^{-1},&\cdots\\[0.2cm]
w_{2}\cdot s^{-1},&\cdots\\[0.3cm]
\vdots&\vdots\\[0.3cm]
w_{d-1}\cdot s^{-1},&\cdots
\end{pmatrix}\quad{\rm where}~w=(w_1,\cdots,w_{d-1})\in\mathbb{R}^{d-1}$$
and
$$k(\gamma)n_ua_r=
\begin{pmatrix}
\left(1+\frac{|u|^2}{2}\right)\frac{r+r^{-1}}{2}-\frac{|u|^2}{2}
\frac{r-r^{-1}}{2},&\quad\cdots &\\[0.5cm]
\left(u_{11}\,\frac{|u|^2}{2}+\sum\limits_{i=2}^nu_{1\,i}u_{i-1}\right)
\frac{r+r^{-1}}{2}+\left[u_{11}\left(1-\frac{|u|^2}{2}\right)-
\sum\limits_{i=2}^nu_{1\,i}u_{i-1}\right]\frac{r-r^{-1}}{2},&\quad\cdots&\\[0.5cm]
\left(u_{21}\,\frac{|u|^2}{2}+\sum\limits_{i=2}^nu_{2\,i}u_{i-1}\right)
\frac{r+r^{-1}}{2}+\left[u_{21}\left(1-\frac{|u|^2}{2}\right)-
\sum\limits_{i=2}^nu_{2\,i}u_{i-1}\right]\frac{r-r^{-1}}{2},&\quad\cdots&\\[0.5cm]
\vdots &\quad\vdots&\\[0.5cm]
\left(u_{d1}\,\frac{|u|^2}{2}+\sum\limits_{i=2}^nu_{d\,i}u_{i-1}\right)
\frac{r+r^{-1}}{2}+\left[u_{d1}\left(1-\frac{|u|^2}{2}\right)-
\sum\limits_{i=2}^nu_{d\,i}u_{i-1}\right]\frac{r-r^{-1}}{2},&\quad\cdots&\quad
\end{pmatrix}.$$
Let $k(\gamma)n_ua_r=n_wa_sk$ for some $k\in K$. Then
we have
\begin{equation}\label{dd1}
\frac{s+s^{-1}}{2}+\frac{s^{-1}}{2}|w|^2=\left(1+\frac{|u|^2}{2}\right)\frac{r+r^{-1}}{2}-\frac{|u|^2}{2}
\frac{r-r^{-1}}{2}
\end{equation}
\begin{equation}\label{dd2}
\frac{s-s^{-1}}{2}+\frac{s^{-1}}{2}|w|^2=
\left(u_{11}\,\frac{|u|^2}{2}+\sum\limits_{i=2}^nu_{1\,i}u_{i-1}\right)
\frac{r+r^{-1}}{2}+\left[u_{11}\left(1-\frac{|u|^2}{2}\right)-
\sum\limits_{i=2}^nu_{1\,i}u_{i-1}\right]\frac{r-r^{-1}}{2}
\end{equation}
\begin{equation}\label{dd3}
w_i\,s^{-1}=\left(u_{i+1,\,1}\,\frac{|u|^2}{2}+\sum\limits_{j=2}^nu_{i+1,\,j}u_{j-1}\right)
\frac{r+r^{-1}}{2}+\left[u_{i+1,\,1}\left(1-\frac{|u|^2}{2}\right)-
\sum\limits_{j=2}^nu_{i+1,\,j}u_{j-1}\right]\frac{r-r^{-1}}{2}\\[0.2cm]
\end{equation}
The equalities (\ref{dd1}) and (\ref{dd2}) imply
\begin{equation}\label{s1}
s^{-1}=\frac{1-u_{11}}{2}\,r+\left(\frac{1+u_{11}}{2}+\beta\right)r^{-1} 
\end{equation}
where $\beta=\left(1-u_{11}\right)\frac{|u|^2}{2}-
\sum\limits_{i=2}^nu_{1\,i}u_{i-1}$.
Let $\alpha_i=u_{i+1,\,1}\,\frac{|u|^2}{2}+
\sum\limits_{j=2}^nu_{i+1,\,j}u_{j-1}$, then (\ref{dd3}) reads
\begin{equation}\label{vs}
w_i\,s^{-1}=\frac{u_{i+1,\,1}}{2}\,r+\left(\alpha_i-\frac{u_{i+1,\,1}}{2}\right)r^{-1},\quad
1\leqslant i\leqslant d-1.
\end{equation}
By the assumption that $\gamma w=n_{v_1}a_{s_1}{\cdot}o$, we have $v_1=(w_0+w)r_0$, $s_1=r_0s$. 
For any $1\leqslant i\leqslant d-1$, denote
$$m_i=w_{0\,i}\,\frac{1-u_{11}}{2}+\frac{u_{i+1,\,1}}{2},\qquad 
n_i=w_{0\,i}\left(\frac{1+u_{11}}{2}+\beta\right)
+\left(\alpha_i-\frac{u_{i+1,\,1}}{2}\right).$$
Then the computation with the above terms shows that
$$f_{\gamma}(u,\,r):=\left|\frac{v_1}{s_1}\right|_{\geqslant n}^2+1=\left|\frac{w_0+w}{s}\right|_{\geqslant n}^2+1=M(\gamma)r^2+
N_u(\gamma)r^{-2}+
Q_u(\gamma)$$
where
\begin{equation}\label{MNQ}
M(\gamma)=\sum\limits_{i=n}^{d-1}
m_i^2,\quad
N_u(\gamma)=
\sum\limits_{i=n}^{d-1}
n_i^2,\quad
Q_u(\gamma)=1+2
\sum\limits_{i=n}^{d-1}
m_in_i.
\end{equation}
Now we have:
$$J_{\gamma}=2^n\left(\sqrt{\frac{\pi}{2\mu}}\,\right)^{n-1}
\int\limits_{\mathbb{R}^{n-1}
\times\mathbb{R}_+}
\left(\sqrt{f_{\gamma}(u,\,r)}\right)^{\nu}
K_{\nu}\left(\mu\sqrt{f_{\gamma}(u,\,r)}\right)s_1^{\nu+\rho_0}
r^{\bar{\nu}+\rho_0-n}drdu.$$
Write 
$f_{\gamma}(u,\,r)=
\left(\sqrt{M(\gamma)}\,r-
\frac{\sqrt{N_u(\gamma)}}{r}\right)^2+
2\sqrt{M(\gamma)N_u(\gamma)}+
Q_u(\gamma)$
and define
$$
\delta_u(\gamma):=
2\sqrt{M(\gamma)N_u(\gamma)}+Q_u(\gamma).$$ 
The
parameter $u$ in the subscript of $N$, $Q$ and $\delta$ indicates that these numbers depend on $u$ as
well as $\gamma$.  
Note that $M$ depends only on $\gamma$. For simplicity, we do not write 
$\gamma$, $u$ explicitly in the notations $m_i$, $n_i$.
The number $\delta_u(\gamma)$ has remarkable geometric meaning which can be interpreted from the following inequalities
$$\cosh\left({\rm d}_{G/K}(\gamma
w,\,z)\right)=\frac{|v-v_1|^2+s_1^2}{2s_1t}
+\frac{t}{2s_1}\geqslant2\sqrt{\frac{|v-v_1|^2+s_1^2}{2s_1}\cdot\frac{1}{2s_1}}=
\sqrt{\left|\frac{v-v_1}{s_1}\right|^2+1}\geqslant\sqrt{\left|\frac{v_1}{s_1}\right|_{\geqslant
n}^2+1}.$$
The ``$=$" at the first
inequality can be achieved as $t$ ranges among all positive numbers.
The last step follows from the fact that $v$ ranges among vectors in
$\mathbb{R}^{n-1}$ (so the ``$=$'' can be achieved). 
To be more precise, the two ``=" are simultaneously achieved at 
$$v={\rm Pr}_{n-1}(v_1),\quad t=\sqrt{s_1^2+|v_1|^2_{\geqslant n}}$$
where ${\rm Pr}_{n-1}$ means the projection map $\mathbb{R}^{d-1}
\rightarrow\mathbb{R}^{n-1}$, $v=(v_1,\cdots, v_{d-1})\mapsto (v_1,\cdots, v_{n-1})$.
Since $r$ ranges over all positive numbers, we have
$\left|\frac{v_1}{s_1}\right|_{\geqslant
n}^2+1=f_{\gamma}(u,\,r)\geqslant\delta_u(\gamma)$ where ``$=$" can be
obtained when $\sqrt{M(\gamma)}\,r-\frac{\sqrt{N_u(\gamma)}}{r}=0$,
i.e., $r=\sqrt{\frac{N_u(\gamma)}{M(\gamma)}}$ if $M(\gamma)\ne 0$, 
or $r=\infty$ if $M(\gamma)=0$. So $\delta_u(\gamma)$ measures the
minimal hyperbolic distance between
the two submanifolds $\gamma n_uA{\cdot}o$
and $G_0/K_0$:
\begin{equation}\label{gmd}
\inf\limits_{a\in A\cdot
o,\,z\in G_0/K_0}{\rm d}_{G/K}(\gamma n_ua{\cdot}o,\,z)=
\textup{arccosh}^+
\left(\sqrt{\delta_u(\gamma)}\right).
\end{equation}
By this formula we know that
$\delta_u(\,\cdot\,)$ is well-defined over
$\Gamma_0\backslash\Gamma$ (but not on $\Gamma/\Gamma_0$).
It is clear from the above discussion that the number
$f_{\gamma}(u,\,r)$ also has remarkable geometric meaning: it
measures the (hyperbolic) distance between the point
$\gamma n_ua_r{\cdot}o$ and the submanifold $G_0/K_0$.
More precisely,
$$
\inf\limits_{z\in G_0/K_0}{\rm d}_{G/K}(\gamma n_ua_r{\cdot}
o,\,z)
=\textup{arccosh}^+
\left(\sqrt{f_{\gamma}(u,\,r)}\right).
$$ 

The rest of this section is devoted to estimating $\Sigma_1$. The crucial ingredient in our argument is that, for each class
$\tilde{\gamma}\in
\Gamma_0\backslash \Gamma/\Gamma_0
\smallsetminus\{\tilde{1}\}$
we shall carefully choose a representative element (which satisfies some universal properties) and deduce the estimate
to a lattice counting problem. We need some technical conclusions  whose proof will be postponed to Sect.\,\ref{2p}.
The first result to be used is 
\begin{prop}\label{x2}
Let $n\geqslant 2$.
For any fixed $u\in\mathbb{R}^{n-1}$, 
we can find a representative element $\gamma$ in each class $\tilde{\gamma}\in
\Gamma_0\backslash \Gamma/\Gamma_0
\smallsetminus\{\tilde{1}\}$
such that the following hold
\begin{itemize}
\item $M(\gamma)N_u(\gamma)
\geqslant c_1>0$ where $c_1$ is independent of $u$ and the representative elements $\gamma$'s.
\item
$\delta_u(\gamma)$ achieves its minimal value
at some $u_{\gamma}\in\mathcal{F}
\subset
\mathbb{R}^{n-1}$
where $\mathcal{F}$ is a fixed compact domain 
independent of $\gamma$'s.
\item $M(\gamma)>c_2>0$ where $c_2$ is independent of $\gamma$'s.
\end{itemize}
\end{prop}
\noindent 
In what follows we shall select 
the representative elements that satisfy the three  
properties in this proposition. 
Let
$x=\sqrt{M(\gamma)}\,r-\frac{\sqrt{N_u(\gamma)}}{r}$, then
$r=\frac{x+\sqrt{x^2+4\sqrt{M(\gamma)N_u(\gamma)}}}{2\sqrt{M(\gamma)}}$ (since $M(\gamma)\ne0$ by Proposition \ref{x2})
and 
\begin{equation*}
J_{\gamma}=2^n\left(\sqrt{\frac{\pi}{2\mu}}\right)^{n-1}
\int\limits_{\mathbb{R}^{n-1}\times
\mathbb{R}}
F_{\gamma}(u,\,x)dxdu
\end{equation*}
where
\begin{equation}\label{Fg}
F_{\gamma}(u,\,x)=s_1^{\nu+\rho_0}
\left(
\frac{x+\sqrt{x^2+
4\sqrt{M(\gamma)N_u(\gamma)}}}{2\sqrt{M(\gamma)}}\right)
^{\overline{\nu}-\rho_0}
\frac{\left(\sqrt{x^2+\delta_u(\gamma)}
\right)^{\nu}K_{\nu}
\left(\mu\sqrt{x^2+\delta_u(\gamma)}
\right)}
{\sqrt{x^2+4\sqrt{M(\gamma)N_u(\gamma)}}}.\end{equation} 
Recall that 
$$s_1=r_0s=\frac{r_0}{\frac{1-u_{11}}{2}\,r+\left(\frac{1+u_{11}}{2}+\beta\right)r^{-1}}.$$
The term 
$\frac{1+u_{11}}{2}+\beta$ is nonnegative since by Cauchy inequality we have
\begin{eqnarray*}
\frac{1+u_{11}}{2}+\beta
&=&\left\{(1-u_{11})|u|^2-2\sum\limits_{i=2}^{n}u_{1i}u_{i-1}
+(1+u_{11})\right\}\big/2\\[0.2cm]
&\geqslant&
\sqrt{\left(1-u_{11}^2\right)\cdot|u|^2}
-\sum\limits_{i=2}^{n}u_{1i}u_{i-1}\\[0.2cm]
&\geqslant& 
\sqrt{\sum\limits_{i=2}^{n}u_{1i}^2
\cdot
\sum\limits_{i=2}^{n}u_{i-1}^2}
-\sum\limits_{i=2}^{n}u_{1i}u_{i-1}\\[0.2cm]
&\geqslant&0
\end{eqnarray*}
The $``="$ holds if and only $u_{1i}=0$ ($n+1\leqslant i\leqslant d$), $u_{i-1}=t_0\cdot u_{1i}$ ($2\leqslant i\leqslant n$) for some constant $t_0$, and $(1-u_{11})|u|^2=1+u_{11}$. These conditions lead to 
at most one solution of $u_{\gamma}$ (up to $\pm1$) for any given $\gamma$. Since we are doing integration along $u$, the possible solution $u_{\gamma}$ can be neglected. Thus,
$$s_1\leqslant
\frac{r_0}{\sqrt{(1-u_{11})(1+u_{11})+2(1-u_{11})\beta}}.$$
In Sect.\,\ref{2p} we shall show 
\begin{prop}\label{u111}
$\sup\limits_{\gamma\in
\Gamma\smallsetminus
\Gamma_0}\big|u_{11}(\gamma)\big|<1$.
\end{prop}
\noindent This proposition indicates that $(1-u_{11})(1+u_{11})+2(1-u_{11})\beta$ 
is a polynomial of degree 2 (with respect to each variable $u_i$). Hence, the denominator of $s_1$ (i.e., $s^{-1}$) grows (at least) polynomially with degree $1$ and positive minimum value \big(as  $\frac{1-u_{11}}{2}\,r+\left(\frac{1+u_{11}}{2}+\beta\right)r^{-1}$ is strictly positive if we neglect $u_{\gamma}$\big).
Multiplying proper $\gamma_0\in\Gamma_{00}$ to the left side of $\gamma$ if necessary, we assume that $r_0$ lies in a compact interval in $\mathbb{R}_+$.  

\noindent
{\bf Case 1.} First we assume that $\big|\Gamma_0\backslash \Gamma/\Gamma_0\big|<\infty$. The case $\big|\Gamma_0\backslash \Gamma/\Gamma_0\big|=\infty$
will be treated later. When $\mu$ is large, $\mu\sqrt{x^2+\delta_u(\gamma)}$ is
also very large since
$\delta_u(\gamma)\geqslant1$. 
By the well-known asymptotic of $K$-Bessel function: 
\begin{equation}\label{kasymp}
K_{z}(x)\sim
\sqrt{\frac{\pi}{2x}}\,e^{-x},\quad{\rm as}~ x\rightarrow\infty
\end{equation} 
the function
$X^{\nu}K_{\nu}\left(
\mu X\right)$ decreases with respect to $X$ when $\mu$ is large. Hence
$$\left|\left(\sqrt{x^2+\delta_u(\gamma)}
\right)^{\nu}K_{\nu}
\left(\mu\sqrt{x^2+\delta_u(\gamma)}
\right)\right|\leqslant
\left(\sqrt{x^2+1}
\right)^{{\rm Re}\,\nu}K_{\nu}
\left(\mu\sqrt{x^2+1}
\right)$$ as $\mu\rightarrow\infty$.
Since $K_{\nu}
\left(\mu\sqrt{x^2+1}
\right)\sim K_{{\rm Re}\,\nu}
\left(\mu\sqrt{x^2+1}
\right)$, we have
$$\left|\int_{\mathbb{R}}
\left(\sqrt{x^2+\delta_u(\gamma)}
\right)^{\nu}K_{\nu}
\left(\mu\sqrt{x^2+\delta_u(\gamma)}
\right)dx\right|\ll
\int_{\mathbb{R}}
\left(\sqrt{x^2+1}
\right)^{{\rm Re}\,\nu}K_{{\rm Re}\,\nu}
\left(\mu\sqrt{x^2+1}
\right)dx.$$
Let $y=x^2+1$ in right hand side of
the above integral and apply the formula 6.592.12 of 
\cite{gr}
\begin{equation*}
\int_1^{\infty}x^{-\frac{b}{2}}(x-1)^{c-1}K_{z}
\left(a\sqrt{x}\right)dx=
2^{c}\,\Gamma(c)a^{-c}K_{b-c}(a),\quad \textup{Re}(a)>0,\,\,\textup{Re}(c)>0,
\end{equation*}
we get
\begin{equation}\label{k111}
\int_0^{\infty}
\left(\sqrt{x^2+\delta_u(\gamma)}
\right)^{{\rm Re}\,\nu}K_{{\rm Re}\,\nu}
\left(\mu\sqrt{x^2+\delta_u(\gamma)}
\right)dx=\frac{\Gamma(1/2)}{\sqrt{2}}
\mu^{-1/2}
K_{-{\rm Re}\,\nu-1/2}(\mu).
\end{equation}
By Cauchy inequality, we have $$\sqrt{M(\gamma)N_u(\gamma)}\geqslant
2\sum\limits_{i=n}^{d-1}m_in_i.$$
The term $n_i$ expands as follows 
\begin{equation}\label{nexp}
n_i=|u|^2\,\frac{w_{0i}(1-u_{11})+u_{i+1,\,1}}{2}+\sum\limits_{j=2}^n\left(u_{i+1,\,j}-w_{0i}u_{1j}\right)u_{j-1}+w_{0i}\frac{1+u_{11}}{2}-\frac{u_{i+1,\,1}}{2}.
\end{equation}
Note that $m_i=\frac{w_{0i}(1-u_{11})+u_{i+1,\,1}}{2}$. Therefore,
\begin{equation}\label{mns}
\sum\limits_{i=n}^{d-1}
m_in_i=|u|^2\,\underbrace{
\sum\limits_{i=n}^{d-1}
m_i^2}_{=M(\gamma)}+
\sum\limits_{i=n}^{d-1}
\sum\limits_{j=2}^{n}m_i\left(u_{i+1,\,j}-w_{0i}u_{1j}\right)u_{j-1}+
\sum\limits_{i=n}^{d-1}
m_i\left(w_{0i}\frac{1+u_{11}}{2}-\frac{u_{i+1,\,1}}{2}\right).\end{equation}
This means that $\sqrt{M(\gamma)N_u(\gamma)}$ grows  polynomially with degree $2$ (with respect to each variable $u_i$) with a positive minimum value for all $\gamma\notin\Gamma_0$ (by Proposition \ref{x2}).

When $x\geqslant 0$, one has 
$$\left(x+\sqrt{x^2+4\sqrt{M(\gamma)
N_u(\gamma)}}
\right)^{{\rm Re}\,\nu-\rho_0}\leqslant
\left(2\sqrt[4]{M(\gamma)
N_u(\gamma)}
\right)^{{\rm Re}\,\nu-\rho_0}$$
noting that ${\rm Re}\,\nu-\rho_0\leqslant0$.
As a consequence, 
\begin{equation*}
\left|\left(
\frac{x+\sqrt{x^2+
4\sqrt{M(\gamma)N_u(\gamma)}}}{2\sqrt{M(\gamma)}}\right)
^{\overline{\nu}-\rho_0}
\frac{s_1^{\nu+\rho_0}}
{\sqrt{x^2+4
\sqrt{M(\gamma)N_u(\gamma)}}}\right|\ll
\left(
\sqrt[4]{M(\gamma)N_u(\gamma)}
\right)^{-(\rho_0-
{\rm Re}\,\nu+1)}\left(s^{-1}\right)^{-(\rho_0+
{\rm Re}\,\nu)}.
\end{equation*}
We have known that both $\sqrt[4]{M(\gamma)N_u(\gamma)}$ and $s^{-1}$ grow (at least) polynomially with degree 1 (with respect to each variable $u_i$) with positive minimum value. Since $\rho_0=\frac{n-1}{2}>0$ for $n\geqslant 2$, the following integral  $$\int_{\mathbb{R}^{n-1}}\left(
\sqrt[4]{M(\gamma)N_u(\gamma)}
\right)^{-(\rho_0-
{\rm Re}\,\nu+1)}(s^{-1})^{-(\rho_0+
{\rm Re}\,\nu)}du$$
converges
as the function inside 
is positive and has polynomial degree (at most) $$-(\rho_0-{\rm Re}\,\nu+1)-
(\rho_0+{\rm Re}\,\nu)=-2\rho_0-1<-1.$$
Note that the term $\left(2\sqrt{M(\gamma)}
\right)^{\rho_0-\overline{\nu}}$
in $F_{\gamma}(u,x)$ does not give essential contribution to the upper bound 
since $\big|\Gamma_0\backslash \Gamma_0/\Gamma_0\big|<\infty$. In view of 
(\ref{k111}), we get
\begin{equation}\label{SITU1}
\int_{\mathbb{R}^{n-1}}
\int_0^{\infty}
F_{\gamma}(u,\,x)dxdu
\ll\mu^{-1/2}K_{-{\rm Re}\,\nu-1/2}(\mu)\ll \mu^{-1}e^{-\mu}.\end{equation}

When $x\leqslant-\omega$ with any $\omega\in(0,1)$, one has
\begin{eqnarray*}
&&\left(x+\sqrt{x^2+4\sqrt{M(\gamma)
N_u(\gamma)}}
\right)^{{\rm Re}\,\nu-\rho_0}
\notag\\[0.2cm]
&=&\left(\frac{\sqrt{x^2+
4\sqrt{M(\gamma)N_u(\gamma)}}-x}{4\sqrt{M(\gamma)N_u(\gamma)}}
\right)^{\rho_0-{\rm Re}\,\nu}\notag\\[0.2cm]
&=&(-x)^{\rho_0-{\rm Re}\,\nu}
\left(
\frac{\sqrt{1+\frac{
4\sqrt{M(\gamma)N_u(\gamma)}}{x^2}}+1}{4\sqrt{M(\gamma)N_u(\gamma)}}
\right)^{\rho_0-{\rm Re}\,\nu}\notag\\[0.2cm]
&=&(-x)^{\rho_0-{\rm Re}\,\nu}
\left(\sqrt{\frac{1}{16M(\gamma)N_u(\gamma)}+\frac{1}{4x^2\sqrt{M(\gamma)N_u(\gamma)}}}
+\frac{1}{4\sqrt{M(\gamma)N_u(\gamma)}}
\right)^{\rho_0-{\rm Re}\,\nu}
\notag\\[0.2cm]
&\overset{(\sharp)}{\leqslant}&
(-x)^{\rho_0-{\rm Re}\,\nu}
\left(
\sqrt{\frac{1}{16M(\gamma)N_u(\gamma)}}+
\sqrt{\frac{1}{4x^2\sqrt{M(\gamma)N_u(\gamma)}}}+
\frac{1}{4\sqrt{M(\gamma)N_u(\gamma)}}
\right)^{\rho_0-{\rm Re}\,\nu}
\notag\\[0.2cm]
&=&(-x)^{\rho_0-{\rm Re}\,\nu}\left(2\sqrt[4]{M(\gamma)N_u(\gamma)}
\right)^{-(\rho_0-{\rm Re}\,\nu)}
\left(\frac{1}{-x}+
\frac{1}{\sqrt[4]{M(\gamma)
N_u(\gamma)}} 
\right)^{\rho_0-{\rm Re}\,\nu}
\notag\\[0.2cm]
&\leqslant&(-x)^{\rho_0-{\rm Re}\,\nu}\left(2\sqrt[4]{M(\gamma)N_u(\gamma)}
\right)^{-(\rho_0-{\rm Re}\,\nu)}
\left(\frac{1}{-x}+
\frac{1}{\sqrt[4]{c_1}} 
\right)^{\rho_0-{\rm Re}\,\nu}
\notag\\[0.2cm]
&\ll&
(-x)^{\rho_0-{\rm Re}\,\nu}\left(2\sqrt[4]{M(\gamma)N_u(\gamma)}
\right)^{-(\rho_0-{\rm Re}\,\nu)}
\left(\frac{1}{\omega}
\right)^{\rho_0-{\rm Re}\,\nu}
\end{eqnarray*}
provided that $\omega$ is small enough.
At the step ($\sharp$) we have used
the inequality $\sqrt{a+b}\leqslant\sqrt{a}+\sqrt{b}$.
Thus, 
\begin{multline*}
\int_{\mathbb{R}^{n-1}}
\int_{-\infty}^{-\omega}
F_{\gamma}(u,\,x)dxdu\ll
\int_{\mathbb{R}^{n-1}}
\int_{-\infty}^{-\omega}
(-x)^{\rho_0-{\rm Re}\,\nu}\left(2\sqrt[4]{M(\gamma)N_u(\gamma)}
\right)^{-(\rho_0-{\rm Re}\,\nu)}
\frac{1}{\omega^{\rho_0-{\rm Re}\,\nu}}\\[0.2cm]
\times \left(2\sqrt[4]{M(\gamma)N_u(\gamma)}
\right)^{-1}
(s^{-1})^{-(\rho_0+{\rm Re}\,\nu)}
\left(\sqrt{x^2+\delta_u(\gamma)}
\right)^{\nu}K_{\nu}
\left(\mu\sqrt{x^2+\delta_u(\gamma)}
\right)
dxdu
\end{multline*} 
As before, the integral
$$\int_{\mathbb{R}^{n-1}}\left(
\sqrt[4]{M(\gamma)N_u(\gamma)}
\right)^{-(\rho_0-
{\rm Re}\,\nu+1)}(s^{-1})^{-(\rho_0+
{\rm Re}\,\nu)}du$$
converges. 
Since $K_{\nu}\left(\mu
\sqrt{x^2+1}
\right)\sim\sqrt{\frac{\pi}{2\mu
\sqrt{x^2+1}}}e^{-\mu
\sqrt{x^2+1}}$ as $\mu\rightarrow\infty$,
the ratio of $K_{\nu}\left(\mu
\sqrt{x^2+1}
\right)$ at $x=0$ and $\mu^{-1/2}$ is asymptotically equal to
$$\frac{1}{\sqrt[4]{1+\mu^{-1}}}{\cdot}
\frac{e^{-\mu}}{e^{-\mu
\sqrt{\mu^{-1}
+1}}}$$
which converges to $e^{1/2}$ as $\mu\rightarrow\infty$. 
Let $\omega=\mu^{-1/2}$, then
\begin{multline*}
\int_{-\infty}^{-\mu^{-1/2}}
(-x)^{\rho_0-{\rm Re}\,\nu}\left(\sqrt{x^2+1}
\right)^{\nu}K_{\nu}
\left(\mu\sqrt{x^2+1}
\right)dx\\[0.2cm]
\asymp\,
\int_{-\infty}^0
(-x)^{\rho_0-{\rm Re}\,\nu}\left(\sqrt{x^2+1}
\right)^{\nu}K_{\nu}
\left(\mu\sqrt{x^2+1}
\right)dx,\quad{\rm as}~\mu\rightarrow\infty.\end{multline*}
Hence, we have
\begin{eqnarray}
&&\int_{\mathbb{R}^{n-1}}
\int_{-\infty}^0
F_{\gamma}(u,\,x)dxdu\notag\\[0.2cm]
&\asymp&\int_{\mathbb{R}^{n-1}}
\int_{-\infty}^{-\mu^{-1/2}}
F_{\gamma}(u,\,x)dxdu\notag\\[0.2cm]
&\ll&
\mu^{(\rho_0-{\rm Re}\,\nu)/2}\int_{-\infty}^{-\mu^{-1/2}}
(-x)^{\rho_0-{\rm Re}\,\nu}\left(\sqrt{x^2+
\delta_u(\gamma)}
\right)^{{\rm Re}\,\nu}K_{{\rm Re}\,\nu}
\left(\mu\sqrt{x^2+\delta_u(\gamma)}
\right)dx\notag\\[0.2cm]
&<&
\mu^{(\rho_0-{\rm Re}\,\nu)/2}\int_{-\infty}^0
(-x)^{\rho_0-{\rm Re}\,\nu}\left(\sqrt{x^2+1}
\right)^{{\rm Re}\,\nu}K_{{\rm Re}\,\nu}
\left(\mu\sqrt{x^2+1}
\right)dx\notag\\[0.2cm]
&\overset{(\ast)}{=}&\mu^{(\rho_0-{\rm Re}\,\nu)/2}\cdot2^{(\rho_0-{\rm
Re}\,\nu-1)/2}\mu^{-
(\rho_0-{\rm
Re}\,\nu+1)/2}\Gamma\left(
\frac{\rho_0-{\rm
Re}\,\nu+1}{2}\right)
K_{-{\rm Re}\,\nu-(\rho_0-{\rm
Re}\,\nu+1)/2}\big(\mu\big)\notag\\[0.2cm]
&\ll&\mu^{-1}e^{-\mu}\label{SITU2}
\end{eqnarray}
where the step $(\ast)$ is computed in the same way with (\ref{k111}). By (\ref{SITU1}) and (\ref{SITU2}),  $$\int_{\mathbb{R}^{n-1}}
\int_{\mathbb{R}} 
F_{\gamma}(u,\,x)dxdu\ll\mu^{-1} e^{-\mu}$$
which proves (\ref{error}).

\medskip
\noindent 
{\bf Case 2.}
Now we deal with the case $\big|\Gamma_0\backslash \Gamma_0/\Gamma_0\big|=\infty$.  
By (\ref{Fg}) and Proposition \ref{x2} we have
\begin{multline}\label{temp1}
\big|J_{\gamma}\big|\ll
\mu^{-(n-1)/2} \int_{
\mathbb{R}^{n-1}}
\int_{\mathbb{R}}
s^{\rho_0+{\rm Re}\,\nu}
\left(2\sqrt{M(\gamma)}
\right)^{\rho_0-{\rm Re}\,\nu}
\left(x+\sqrt{x^2+
4\sqrt{M(\gamma)
N_u(\gamma)}}
\right)^{{\rm Re}\,\nu-
\rho_0}\\[0.2cm]
\times K_{\nu}
\left(\mu\sqrt{x^2+\delta_u(\gamma)}
\right)dxdu.
\end{multline}
By (\ref{kasymp}) and the inequality $\sqrt{a+b}\geqslant
\frac{\sqrt{2}}{2}\left(\sqrt{a}+\sqrt{b}\right)$,
\begin{equation}\label{temp2}
K_{\nu}\left(\mu
\sqrt{x^2+\delta_u(\gamma)}
\right)\sim
\sqrt{\frac{\pi}{2\mu\sqrt{x^2+\delta_u
(\gamma)}}}e^{-\mu\sqrt{x^2+
\delta_u(\gamma)}}
\ll\mu^{-1/2}e^{-\frac{\sqrt{2}}{2}\mu|x|}e^{-\frac{\sqrt{2}}{2}\mu\sqrt{
\delta_u(\gamma)}}.
\end{equation}
Combining (\ref{temp1}) and (\ref{temp2}) yields
\begin{equation*}\label{temp3}
\big|J_{\gamma}\big|
\ll\mu^{-n/2}\int_{
\mathbb{R}^{n-1}}
\int_{\mathbb{R}}
s^{\rho_0+
{\rm Re}\,\nu}e^{-\frac{\sqrt{2}}{2}\mu|x|}
\left(x+\sqrt{x^2+
4\sqrt{M(\gamma)
N_u(\gamma)}}
\right)^{{\rm Re}\,\nu-
\rho_0}
\left(2\sqrt{M(\gamma)}
\right)^{\rho_0-{\rm Re}\,\nu}e^{-\frac{\sqrt{2}}{2}\mu\sqrt{
\delta_u(\gamma)}}dxdu.
\end{equation*}
The integral 
$$\int\limits_{\mathbb{R}}
s^{\rho_0+{\rm Re}\,\nu}e^{-\frac{\sqrt{2}}{2}\mu|x|}\left(x+\sqrt{x^2+
4\sqrt{M(\gamma)
N_u(\gamma)}}
\right)^{{\rm Re}\,\nu-
\rho_0}dx$$ converges
and is uniformly upper bounded by a constant since the function inside has exponential decay
and $\mu$ is large (note that $s$ is a positive rational function of $x$). 

As 
$\delta_u(\gamma)$ grows polynomially with respect to
the variables $\sqrt{M(\gamma)}\,u_i$,
for any $\varepsilon\in
(0,1)$
we have
\begin{eqnarray*}
&&
\sum\limits_{\tilde{\gamma}
\in\Gamma_0\backslash \Gamma/\Gamma_0\smallsetminus
\{\tilde{1}\}}\,\,
\int_{
\mathbb{R}^{n-1}}
\left(2\sqrt{M(\gamma)}
\right)^{\rho_0-{\rm Re}\,\nu}e^{-\frac{\sqrt{2}}{2}\mu
\sqrt{\delta_u(\gamma)}
}du\\[0.2cm]
&=&\sum\limits_{\tilde{\gamma}
\in\Gamma_0\backslash \Gamma/\Gamma_0\smallsetminus
\{\tilde{1}\}}\,\,\int_{
\mathbb{R}^{n-1}}
\left(2\sqrt{M(\gamma)}
\right)^{\rho_0-{\rm Re}\,\nu}e^{-\frac{\sqrt{2}}{2}\mu
\sqrt{\varepsilon^2
\delta_u(\gamma)+
\left(1-\varepsilon^2\right)
\delta_u(\gamma)}
}du\\[0.2cm]
&\leqslant&\sum\limits_{\tilde{\gamma}
\in\Gamma_0\backslash \Gamma/\Gamma_0\smallsetminus
\{\tilde{1}\}}\,\,
\int_{
\mathbb{R}^{n-1}}
\left(2\sqrt{M(\gamma)}
\right)^{\rho_0-{\rm Re}\,\nu}e^{-\frac12\mu
\varepsilon
\sqrt{
\delta_u(\gamma)}}\cdot
e^{-\frac12\mu
\sqrt{1-\varepsilon^2}
\sqrt{\delta_u(\gamma)}}
du
\\[0.2cm]
&\ll&\int_{
\mathbb{R}^{n-1}}
\left(\sqrt{M(\gamma)}
\right)^{-\rho_0-{\rm Re}\,\nu}
e^{-\frac12\mu
\varepsilon
\sqrt{
\delta_u(\gamma)}}
d\left(\sqrt{M(\gamma)}\,
u\right)\times
\sum\limits_{\tilde{\gamma}
\in\Gamma_0\backslash \Gamma/\Gamma_0\smallsetminus
\{\tilde{1}\}}\,\,
e^{-\frac12\mu
\sqrt{1-\varepsilon^2}
\sqrt{\delta_{u_{_{\ast}}}(\gamma)}}
\end{eqnarray*}
for some $u_{_{\ast}}\in\mathcal{F}$. At the second step we have used the 
inequality $\sqrt{a+b}\geqslant
\frac{\sqrt{2}}{2}(\sqrt{a}+\sqrt{b})$. At the last step we have used Proposition \ref{x2}.
Note that
$\left(\sqrt{M(\gamma)}
\right)^{-\rho_0-{\rm Re}\,\nu}$ is uniformly upper bounded since $M(\gamma)>c_2>0$. Thus, the integral in the last step
converges. It follows that
\begin{equation}\label{jgl}
J_{\gamma}\ll\mu^{-n/2} e^{-\frac12\mu
\sqrt{1-\varepsilon^2}
\sqrt{\delta_{u_{_{\ast}}}(\gamma)}}
\end{equation}
In Sect.\,\ref{2p} we shall show
\begin{prop}\label{x1}
For any fixed $u\in\mathbb{R}^{n-1}$
and sequence $\Lambda_u\subset
\Gamma$ where each element in
$\Lambda_u$ represents exactly one double 
class in
$\Gamma_0\backslash
\Gamma/\Gamma_0$, there exist a positive number $c_3$ such that
$$\pi_{\Lambda_u}(x):=\#\{\gamma\in\Lambda_u\,|\,\delta_u(\gamma)\leqslant x\}\leqslant c_3\cdot x^{(d-n)/2},\quad\textup{as}~x\rightarrow\infty.$$
\end{prop}
\noindent 
We may arrange the order of
the elements in $\Lambda_u$ and get $\Lambda_u=\{\gamma_{_j}\}$ such that  
$\delta_u(\gamma_{_j})$ is nondecreasing as $j$ grows. Then we have
\begin{prop}\label{dltm}
$\delta_u(\gamma_{_j})\gg j^{\frac{1}{(d-n)/2+\beta}}$ 
for any fixed $\beta>0$ and $u$.
\end{prop}
\begin{proof}
It suffices to show that $\delta
\left(
\tilde{\gamma}_{\left[j^{(d-n)/2+\beta}\right]}\right)\gg j$, where $[x]$ means taking the maximal integer that does not exceed $x$. Assume that
there exists a sequence $\{j_{_i}\}_{i=1}^{\infty}$ such that $j_{_i}$ increases as 
$i$ grows and $$\frac{\delta
\left(
\tilde{\gamma}_{_{\left[
j_{_i}^{(d-n)/2+\beta}\right]}}\right)}{j_{_i}}\rightarrow0$$ as 
$i\rightarrow\infty$. 
Then all $\tilde{\gamma}_{_{\left[
j_{_i}^{(d-n)/2+\beta}\right]}}$ are 
contained in the set $\left\{\tilde{\gamma}\,|\,\delta(\tilde{\gamma})\leqslant j_{_i}\right\}$. This implies that $\pi(j_{_i})\geqslant\left[
j_{_i}^{(d-n)/2+\beta}\right]$ since $\delta(\gamma_{_j})$ is nondecreasing 
as $j$ grows, contradicting Proposition \ref{x1}.
\end{proof}
Let $u=u_{\ast}$, then by Proposition \ref{dltm} there exists $j_{_0}>0$ such that $\delta_{u_{\ast}}(\gamma_{_j})\geqslant \frac{4}{1-\varepsilon^2}\,j^{2/d}$
for any $j\geqslant j_{_0}$. Denote $J_j=J_{\gamma_{_j}}$. 
By (\ref{jgl}) we have
$$\sum\limits_{j=j_{_0}}^{\infty}
J_j\ll
\mu^{-n/2}
\sum\limits_{j=j_{_0}}^{\infty}
e^{-\mu\,j^{1/d}}< \mu^{-n/2}
\int_{1}^{\infty}e^{-\mu\,x^{1/d}}dx
=d\mu^{-n/2}\int_{1}^{\infty}e^{-\mu\,y}y^{d-1}dy$$
where we made the variable exchange $x^{1/d}\rightarrow y$ in the last step. 
Integration by parts shows that the integral on the right hand side is upper bounded by $\mu^{-1}e^{-\mu}$.
Hence, $$\sum\limits_{j=j_{_0}}^{\infty}
J_j\ll
\mu^{-(n+2)/2}e^{-\mu}.$$
As for those $J_{j}$ where $1\leqslant j\leqslant j_{_0}$, we apply the argument
that has been done for the
case $\big|\Gamma_0
\backslash \Gamma/\Gamma_0
\big|<\infty$ and get
$\sum\limits_{j=1}^{j_{_0}}
J_j\ll 
\mu^{-(n+1)/2}e^{-\mu}$, thereby we
have shown (\ref{error}).

Putting the data on geometric and spectral sides
together, we get
\begin{eqnarray*}\sum\limits_{j=0}^{\infty}2^d\left(\sqrt{\frac{\pi}{2\mu}}\,\right)^{d-1}\, K_{\nu_j}(\mu)\,\big|P_Y(\phi_j,\psi)\big|^2&=&b_{\nu}\cdot
\left(\sqrt{\frac{\pi}{2\mu}}\,\right)^{n-1}K_{\bar{\nu}}(\mu)+
\mathcal{O}\left(e^{-\mu}
\mu^{-(n+1)/2}\right)\end{eqnarray*}
In view of the asymptotic (\ref{kasymp}) of $K$-Bessel function,  
multiplying $2^{-d}\left(\sqrt{\frac{2\mu}{\pi}}\right)^{n}e^{\mu}$
on both sides of this formula and taking the limitation
$\mu\rightarrow\infty$ yields
\begin{equation}\label{main}
\lim\limits_{\mu\rightarrow\infty}
e^{\mu}\left(\sqrt{\frac{\pi}{2\mu}}\,\right)^{d-n-1}
\sum\limits_{j=0}^{\infty}\, K_{\nu_j}(\mu)\,\big|P_Y(\phi_j,\psi)\big|^2=c_{\nu}.
\end{equation}
where $c_{\nu}=2^{-d}b_{\nu}$.
The K-Bessel function $K_{\nu_i}(\mu)$ is positive
for $\nu_j\in\,i\,\mathbb{R}$, so the right hand side of (\ref{main}) is nonnegative. Since $b_{\nu}\ne0$ (see Sect.\,\ref{mt}), the scalar 
$c_{\nu}$ is positive. 
\begin{proof}[Proof of Theorem \ref{thm}]
Assume that there are only finitely many $\phi_j$
such that $P_Y(\phi_j,\psi)\ne0$. By the aysmptotic formula (\ref{kasymp}),
the left hand side of (\ref{main})
is equal to zero, a contradiction as $c_{\nu}>0$. 
\end{proof}

\subsection{Proofs of  Proposition \ref{x2}, \ref{u111},  \ref{x1}}\label{2p} 
In this section we prove the propositions that have been used in the previous section. The commutativity relations listed below, as well as (\ref{commu}), will be used frequently:
\begin{equation}\label{commu2}
{\rm diag}(1,1,k)\cdot n_u=n_{uk^T}\cdot{\rm diag}(1,1,k),\quad u\in\mathbb{R}^{d-1},\quad k\in SO(d-1).
\end{equation}
\begin{equation}\label{commu3}
{\rm diag}(1,-1,k^{\prime})\cdot 
a_r=
a_{r^{-1}}\cdot{\rm diag}(1,-1,k^{\prime}),\quad
k^{\prime}\in O(d-1),~\det(k^{\prime})=-1.
\end{equation}
These properties are easy to verify. Under our assumption on $\Gamma_{00}$ (see Sect.\,\ref{sym}), there exists 
$\gamma_0=a_{\ell_0}k_0\in AM_0$ such that $\Gamma_{00}=
\langle\gamma_0\rangle$. It is clear that $\ell_0\ne1$. We might as well assume that $\ell_0>1$.
\begin{lem}\label{ANM}
$u_{11}(\gamma)\ne\pm 1$ for any $\gamma\in\Gamma\smallsetminus
\Gamma_0$. 
\end{lem}
\begin{proof}
Assume that $u_{11}(\gamma)=1$, i.e., $\gamma\in (ANM\cap\Gamma)
\smallsetminus
\Gamma_0$. Write $\gamma=a_{r_0}n_{w_0}k(\gamma)\in ANM$ where $k(\gamma)={\rm diag}(1,1,k)$ with $k\in SO(d-1)$. Multiplying $\gamma_1=a_{r_1}{\rm diag}(1,1,k_1)$, 
$\gamma_2=a_{r_2}{\rm diag}(1,1,k_2)\in\Gamma_{00}$
(where $k_1$, $k_2\in SO(d-1)$) to
the right and left sides of $\gamma$ yields
$$\gamma_2\gamma\gamma_1=
a_{r_2r_0r_1}
n_{w_0k_2^Tr_1^{-1}}\cdot
{\rm diag}(1,1,k_2kk_1).
$$
Here we have use the commutativity relations (\ref{commu}) and (\ref{commu2}). With proper $r_1$, $r_2$,
we may assume that 
$r_2r_0r_1$ lies in $(1,\,\ell_0]$, and 
$\big|w_0k_2^Tr_1^{-1}\big|=
|w_0|r_1^{-1}$ is small enough.
This means that $\gamma_2\gamma\gamma_1$ is 
close to $AM\cap\Gamma=\Gamma_{00}$.
The discreteness of $\Gamma$
implies that $\gamma_2\gamma\gamma_1$ lies in $\Gamma_{00}\subset\Gamma_0$, whence $\gamma\in\Gamma_0$. This contradicts the assumption. The case that $u_{11}(\gamma)=-1$ can be disproved  in
the same fashion where we should use (\ref{commu3}).  
\end{proof}
\begin{proof}[Proof of Proposition \ref{u111}]
If there exists a sequence $\{\gamma_i\}\subset\Gamma
\smallsetminus\Gamma_0$ such that $u_{11}(\gamma_i)\rightarrow1$,
then $\gamma_i$
is close to $ANM\cap\Gamma$. By Lemma Lemma \ref{ANM} and the  discreteness of $\Gamma$,
$\gamma_i$ lies 
in $ANM\cap\Gamma=
\Gamma_0$ for large enough $i$. However, we have assumed that $\gamma\notin\Gamma_0$.
The case $u_{11}(\gamma_i)\rightarrow-1$ can be shown analogously.
\end{proof}
\begin{lem}\label{5a}
The image of $\Gamma_0\backslash\Gamma$ under the map $\Gamma_0\backslash\Gamma\rightarrow\Gamma_0\backslash G$ is discrete and has no accumulation point in $\Gamma_0\backslash G$
with respect to the topology defined by the invariant Riemann metric of $\Gamma_0\backslash G$.
\end{lem}
\begin{proof}
We show the first part of the lemma since the second part follows in the same
way. Assume that the sequence $\{\Gamma_0{\cdot}
\gamma_i\}\subset
\Gamma_0\backslash G$ converges to $\Gamma_0{\cdot}\gamma
$ where $\gamma_i$, $\gamma\in\Gamma$, and $\Gamma_0{\cdot}\gamma_i
\ne\Gamma_0{\cdot}\gamma_j
$ for $i\ne j$. Then there exist a sequence $\{\eta_i\}\subset\Gamma_0$ and a compact neighborhood $W_i$ of $e$ such that $\gamma^{-1}\eta_i\gamma_i\in W_i\cap\Gamma$. By passing to a subsequence if necessary, we may assume that $\left\{
\gamma^{-1}\eta_i\gamma_i
\right\}$ converges. 
The compact neighborhood $W_i$ can be close to $e$
arbitrarily for $i$ large,
which means that $W_i\cap\Gamma=\{e\}$ for $i$ large enough. Hence,
$\gamma^{-1}\eta_i\gamma_i
=e$ 
for $i$ large enough. As a consequence, $\Gamma_0{\cdot}\gamma_i
=\Gamma_0{\cdot}\gamma_j
$ for $i$, $j$ large enough, a contradiction.
\end{proof}
\begin{proof}[Proof of Proposition \ref{x1}]
First we assume that $M(\gamma)N_u(\gamma)\ne0$.
For given $\gamma=
a_{r_0}n_{w_0}
k(\gamma)\in\Lambda_u$ and $u\in\mathbb{R}^{n-1}$ there exist (unique)
$a_r\in A$ and $\gamma_1=a_{r_1}\cdot{\rm diag}(1,1,k_1,k_2)\in\Gamma_{00}$ with $r\in (1,\ell_0]$ and 
$k_1\in O(n-1)$, $k_2\in
O(d-n)$, such that
$\gamma n_u a_{rr_1}{\cdot}o=a_tn_v{\cdot}o$ where 
$|v|_{\geqslant
n}^2+1=\delta_u(\gamma)$. Here,
$rr_1$ is the unique positive solution of the equation
$\sqrt{M(\gamma)}\,x-\frac{\sqrt{N_u(\gamma)}}{x}=0$ (see Sect.\,\ref{et} for the meaning of the notations). Thus,  modulo
$\left\{\ell_{0}^n\,\big|\,n\in\mathbb{Z}\right\}$ (multiplicatively), $r$ and $r_1$ (then, $a_r$ and $\gamma_1$) are unique.
Clearly we have: $\gamma n_u a_{rr_1}{\cdot}
o=\gamma\gamma_1 a_r n_{\left(rr_1\right)^{-1}u
k_1}{\cdot}o$ 
(here we have used (\ref{commu}), (\ref{commu2})). For  
$v=(v_1,v_2,\cdots,v_{d-1})\in
\mathbb{R}^{d-1}$, denote 
$$v_{<n}=(v_1,\,\cdots,\,v_{n-1},\,0,\,\cdots,\,0)\quad{\rm and}\quad v_{\geqslant
n}=(0,\,\cdots,\,0,\,v_n,\,\cdots,\,v_{d-1}).$$
Let $G_0^{\ast}$ denote a fundamental domain of $\Gamma_0\backslash G_0$ in $G_0$ which contains $\Gamma_0{\cdot}e$.
Then there exists (unique)
$\gamma_2\in\Gamma_0$ such that $\gamma_2a_tn_v{\cdot}o
=\gamma_2a_tn_{v_{<n}}\cdot
n_{\geqslant n}{\cdot}o$ such that
$\gamma_2a_tn_{v_{<n}}$ lies in $G_0^{\ast}$. Define
$$\mathbb{R}^{d-1}_n :=\left\{
v_{\geqslant n}\,\big|\,v\in\mathbb{R}^{d-1}\right\}$$ and
$$\Omega_u(x):=\left\{
g_0n_uk\in G\,\big|\,
g_0\in G_0^{\ast},\,\, u\in\mathbb{R}^{n-1}_n 
,\,\,|u|^2+1\leqslant x,\,\,k\in K\right\}\subset G.$$
 Denote by $\gamma^{\ast}(u)$ the element  $\gamma\gamma_1 a_r n_{\left(rr_1\right)^{-1}u
k_1}{\cdot}
o\in\Omega_u(x)$.
When $M(\gamma)N_u(\gamma)=0$, the existence of $\gamma^{\ast}(u)$ in $\Omega_u(x+\varepsilon_0)$ is clear in view of the above argument: if $M(\gamma)=0$, one chooses $\gamma_1=\gamma_0^{\ell}$ for $\ell>0$ large enough; if $N_u(\gamma)=0$, one chooses $\gamma_1=\gamma_0^{\ell}$ for $\ell<0$ and $|\ell|$ large enough. Here $\varepsilon_0$ is a positive, small enough number.
\begin{lem}\label{cl}
$\gamma^{\ast}(u)\ne
\eta^{\ast}(w)$ for any $u$, $w\in\mathbb{R}^{n-1}$ and $\gamma$, $\eta$ of different classes in $\Gamma_0\backslash
\Gamma/\Gamma_0$. \end{lem}
\begin{proof}
Like $\gamma^{\ast}(u)$, we have:
$\eta^{\ast}(w)=
\eta_2\eta\eta_1
a_{\ell}
n_{(\ell\ell_1)^{-1}
w\tau_1}\in
\Omega^{\ast}_x$ for some $\ell\in(1,\ell_0]$,  $\eta_2\in\Gamma_0$ and
$\eta_1=a_{\ell_1}
k_1^{\prime}\in\Gamma_{00}$ where
$k_1^{\prime}=\textup{diag}(1,\,1,\,\tau_1,\tau_2)\in M_0$. If
$\gamma^{\ast}(u)=\eta^{\ast}(w)$, then 
$$a_rn_{(rr_1)^{-1}uk_1}
\cdot
n_{-(\ell\ell_1)^{-1}w\tau_1}\,
a_{\ell^{-1}}=\gamma_1^{-1}
\gamma^{-1}\gamma_2^{-1}
\cdot
\eta_2\eta\eta_1
\in\Gamma\cap
AN_0.$$
It is clear that $AN_0\cap\Gamma\subset \Gamma_0$. Hence,
$\gamma_1^{-1}
\gamma^{-1}\gamma_2^{-1}
\cdot
\eta_2\eta\eta_1\in\Gamma_0$ which
implies that $\gamma$ and $\eta$ are of the same class in
$\Gamma_0\backslash\Gamma/\Gamma_0$, contradicting our
assumption.
\end{proof}
This lemma tells us that, those 
$\gamma^{\ast}(u)\in\Omega_u(x)$ are distinguishable with respect to $\gamma$ (of different classes) and $u$.
A further property is about the discreteness of $\gamma^{\ast}(u)$:
\begin{lem}\label{hi}
For any sequence of pairs
$$\Big\{\big(\gamma_{i1}^{\ast}(u_i),\gamma_{i2}^{\ast}(w_i)\,\big)\,\Big|\,\gamma_{i1},\,\gamma_{i2}\in\Gamma
\smallsetminus\Gamma_0,\,\tilde{\gamma}_{i1}\ne
\tilde{\gamma}_{i2},\,\gamma_{i1}^{\ast}(u_i)\in\Omega_{u_i}(\infty),\,\gamma_{i2}^{\ast}(w_i)\in\Omega_{w_i}(\infty),\,u_i,\,w_i\in\mathbb{R}^{n-1}
\Big\}_{i=1}^{\infty},$$
$\gamma_{i1}^{\ast}(u_i)$ and $\gamma_{i2}^{\ast}(w_i)$ can not be close
enough (as $i\rightarrow\infty$) with respect to the topology of
$G$.
\end{lem}
\begin{proof}
Let
$\gamma_{i1}^{\ast}(u_i)=
\gamma_{i1}^{\prime}
\gamma_{i1}
\gamma_{i1}^{\prime
\prime}a_{r_{i1}}
n_{u_{i1}}\in
\Omega_{u_i}(\infty)$
and
$\gamma_{i2}^{\ast}(w_i)=
\gamma_{i2}^{\prime}
\gamma_{i2}
\gamma_{i2}^{\prime
\prime}a_{r_{i2}}
n_{u_{i2}}\in
\Omega_{w_i}(\infty)$
for some $\gamma_{i1}^{\prime}$,
$\gamma_{i2}^{\prime}\in\Gamma_0$, $\gamma_{i1}^{\prime\prime}$,
$\gamma_{i2}^{\prime\prime}\in\Gamma_{00}$, $r_{i1}$,
$r_{i2}\in(1,\,\ell_0]$ and $u_{i1}$, $u_{i2}\in\mathbb{R}^{n-1} $.
Assume that $\gamma_{i1}^{\ast}(u_i)$ and $\gamma_{i2}^{\ast}(w_i)$ are close
enough as $i\rightarrow\infty$, then
$\left(\gamma_{i1}^{\ast}(u_i)\right)^{-1}\gamma_{i2}^{\ast}(w_i)\rightarrow1$, that is,
$n_{-u_{i1}}a_{r_{i1}^{-1}}
\left(
\gamma_{i1}^{\prime}
\gamma_{i1}
\gamma_{i1}^{\prime
\prime}\right)^{-1}
\left(
\gamma_{i2}^{\prime}
\gamma_{i2}
\gamma_{i2}^{\prime
\prime}\right)
a_{r_{i2}}n_{u_{i2}}\in
U_i$
where $U_i$ is a compact neighborhood of $e$ that can be small enough for
large $i$. 
It follows that 
$\eta_i:=\left(
\gamma_{i1}^{\prime}
\gamma_{i1}
\gamma_{i1}^{\prime
\prime}\right)^{-1}
\left(
\gamma_{i2}^{\prime}
\gamma_{i2}
\gamma_{i2}^{\prime
\prime}\right)$ lies in $V_i:=a_{r_{i1}}
n_{u_{i1}}\,
U_i\,
n_{-u_{i2}}a_{r_{i2}^{-1}}$,
a compact neighborhood of  $a_{r_{i1}}n_{(u_{i1}-u_{i2})}a_{r^{-1}_{i2}}
\in AN_0$ which is contained in $AN_0V$ where $V$ is a fixed compact neighborhood of $e$ (note that $V_i$ is small enough for large $i$). Since $\Gamma_0\backslash G_0$ is compact,  the image of 
$V_i$ in $\Gamma_0\backslash G$ is also compact. This implies that,
passing to a subsequence if necessary we may assume that $\Gamma_0{\cdot}\eta_i$ converges in $\Gamma_0\backslash G$. 
By Lemma \ref{5a}, the sequence $\{\Gamma_0{\cdot}\eta_i\}$ becomes stable for large enough $i$. Hence, $\tilde{\gamma}_i=
\tilde{\gamma}_j$ for large $i$ and $j$, a contradiction.
\end{proof}
By Lemma \ref{cl}, to
count $\pi_u(x)$ it suffices to count the representative elements $\gamma_i^{\ast}(u)$ that lie in $\Omega_u(x)$.  Lemma \ref{hi} tells us that these $\gamma_i^{\ast}(u)$ are discrete and have no accumulation point with respect to the topology of $G$.  The topology of $G$, when restricted to 
$\Omega_u(x)$ is equivalent to the Euclidean topology of $\mathbb{R}^{d-n}$ since the components $G_0^{\ast}$ and $K$ of $\Omega_u(x)$ are compact. Thus, $\pi_u(x)$ is (upper) bounded by the volume
of $\Omega_u(x)$ which is of order $x^{(d-n)/2}$. This 
proves Proposition \ref{x1}.
\end{proof}
\begin{cor}\label{uag}
Assume that
$|\,\Gamma_0\backslash\Gamma/\Gamma_0|=\infty$, then for any fixed $u\in\mathbb{R}^{n-1}$
the unique accumulation point of $\{\delta_u(\gamma_i)\,|\,\tilde{\gamma}_i\ne
\tilde{\gamma}_j\}$ is $\infty$.
\end{cor}
\begin{lem}\label{c2}
If $M(\gamma)N_u(\gamma)=0$, then $Q_u(\gamma)=\delta_u
(\gamma)=1$.
\end{lem}
\begin{proof}
This is clear in view of (\ref{MNQ}).
\end{proof}
\begin{lem}\label{mnz}
For any $\gamma\notin\Gamma_0$ and $u\in\mathbb{R}^{n-1} $,
$M(\gamma)$ and $N_u(\gamma)$ can not be zero simultaneously.
\end{lem}
\begin{proof}
Assume that $M(\gamma)=N_u(\gamma)=0$ for some
$\gamma\notin\Gamma_0$ and $u\in\mathbb{R}^{n-1} $. As before, let
$\gamma n_ua_r=a_{r_0s}
n_{\frac{w_0+w}{s}}k$ (see Sect.\,\ref{et}). Then
$\left|\frac{w_0+w}{s}\right|^2_{\geqslant
n}+1=M(\gamma)r^2+\frac{N_u(\gamma)}{r^2}+Q_u(\gamma)\equiv1$ for any
$r>0$ (see Lemma \ref{c2}). This means that (see (\ref{gmd}))
$$\gamma n_u A{\cdot}o\subset G_0/K_0.$$
As $n_uA{\cdot}o\subset G_0/K_0$, we have $\gamma\in\Gamma_0$, a contradiction.
\end{proof}
\begin{proof}[Proof of Proposition \ref{x2}]
As before (see Sect.\,\ref{et}), 
let $w=n_ua_r{\cdot}o$, $z=n_va_t{\cdot}o\in G_0/K_0$, $\gamma=a_{r_0}
n_{w_0}k(\gamma)$. 
Write $k(\gamma)n_u
a_r{\cdot}o=n_w
a_s{\cdot}o$. Then
$\gamma n_ua_r{\cdot}o=n_{v_1}a_{s_1}
{\cdot}o$ where $v_1=(w_0+w)r_0$, $s_1=r_0s$.

First we show that, for each class $\tilde{\gamma}\ne\tilde{1}$ there exists representative element $\gamma$ such that $M(\gamma)N_u(\gamma)\ne 0$ for any $u\in\mathbb{R}^{n-1}$.
Recall that $\frac{1+u_{11}}{2}+\beta=0$ has at most one solution $u(\gamma)$ of $u$ up to $\pm 1$. Multiplying $\gamma_1\in\Gamma_0$ to the 
left side of the geodesic $C_u:=n_uA{\cdot}o\subset G_0/K_0$ and denote the point $P\in\gamma_1C_u$ as $n_{u^{\prime}}a_{r^{\prime}}{\cdot}o$. Clearly, we can choose
proper $\gamma_1$ such that $u^{\prime}\ne u(\gamma)$ for any $P\in\gamma_1C_u$. Replacing $\gamma$ with $\gamma\gamma_1$ if necessary, we assume that $\frac{1+u_{11}}{2}+\beta\ne 0$. 
Note that the condition $n\geqslant 2$ is necessary in the above discussion because the submanifold $G_0/K_0$ should be large enough to allow inside the geodesic $\gamma_1C_u$ (and $\gamma_2C_u$ in below). If $n=1$, such property does not hold.
As $r\rightarrow\infty$, we have: $$s_1=\frac{r_0}{\frac{1-u_{11}}{2}\,r+\left(\frac{1+u_{11}}{2}+\beta\right)r^{-1}}
\rightarrow 0\quad{\rm and}\quad 
w_i=\frac{\frac{u_{i+1,\,1}}{2}\,r+\left(\alpha_i-\frac{u_{i+1,\,1}}{2}\right)r^{-1}}{\frac{1-u_{11}}{2}\,r+\left(\frac{1+u_{11}}{2}+\beta\right)r^{-1}}
\rightarrow\frac{u_{i+1,1}}{1-u_{11}}=:\tilde{w}_i.$$
Denote $\tilde{w}=(\tilde{w}_1,\cdots,\tilde{w}_{d-1})$. Then $\tilde{w}$ does not depend on $u$ (as $r\rightarrow\infty$). It follows that $\xi_{\gamma}:=
\lim_{r\rightarrow\infty}
\gamma n_ua_r{\cdot}o=
\lim_{r\rightarrow0^+}
n_{(w_0+\tilde{w})r_0}a_{r}{\cdot}o$ is a fixed point lying in the 
boundary $\mathcal{X}$ of $G/K$ where $\mathcal{X}:=\{\lim_{r\rightarrow 0^+}n_wa_r{\cdot}o\,|\,w\in\mathbb{R}^{d-1}\}$.
Note that $\xi_{\gamma}$ is independent of $u$. As $r\rightarrow 0$, we have $s_1\rightarrow0$ and 
$w_i\rightarrow\frac{\alpha_i-\frac{u_{i+1,1}}{2}}{\frac{1+u_{11}}{2}+\beta}=:\mathring{w}_i$. Denote $\mathring{w}=(\mathring{w}_1,\cdots,\mathring{w}_{d-1})$
and $\xi_{\gamma}(u)=\lim_{r\rightarrow0^+}
n_{(w_0+\mathring{w})}
a_r{\cdot}o\in\mathcal{X}$. The geodesic on $G/K$ is either a straight line or a semicircle, each of which is vertical to $\mathcal{X}$ (in the $\mathbb{R}^{d-1}
\times\mathbb{R}_+$ model). Our argument in above shows that $\gamma C_u$ is a semicircle with two end points 
$\xi_{\gamma}$,  $\xi_{\gamma}(u)\in\mathcal{X}$.
Multiply $\gamma_2\in\Gamma_0$ on the left side of the geodesic $C_u$ and assume that $\gamma_2 C_u$ is a semicircle on $G_0/K_0$ with two end points $P_1=\lim_{r\rightarrow 0^+}n_{u_1}a_r{\cdot}o$, $P_2=\lim_{r\rightarrow 0^+}n_{u_2}a_r{\cdot}o$ where $u_1$, $u_2\in\mathbb{R}^{n-1}$.
We can choose $\gamma_2$ such that the resulting  $u_1$, $u_2$ are large enough, then $u(\gamma)\ne v^{\prime}$ for any $n_{v^{\prime}}
a_{t^{\prime}}{\cdot}o\in\gamma_2C_u$. This implies that  $\frac{1+u_{11}}{2}+\beta\ne0$ for $\gamma\gamma_2$ ($u$ is the same with before).
Define $\mathcal{Y}:=\{\lim_{r\rightarrow 0^+}n_wa_r{\cdot}o\,|\,w\in\mathbb{R}^{n-1}\}$, the boundary of $G_0/K_0$. Then $\mathcal{Y}$ is a proper subset of $\mathcal{X}$.

\begin{itemize}
\item
If $M(\gamma)=0$, then $N_u(\gamma)\ne 0$ by Lemma \ref{mnz}. It follows from Lemma \ref{c2} 
that $f_{\gamma}(u,r)=N_u(\gamma)/r^2+1\rightarrow 1=\delta_u(\gamma)$.
The (minimal) distance between the geodesic $\gamma C_u$ and the submanifold $G_0/K_0$, which is $0$, is obtained at the two points $n_{v_1}a_{s_1}{\cdot}o\in\gamma C_u$
and $n_va_t{\cdot}o\in G_0/K_0$ where $v_1=(w_0+\tilde{w})r_0$, $s_1=0^+$, $v={\rm Pr}_{n-1}(v_1)$, $t=\sqrt{s_1^2+|v_1|^2_{\geqslant n}}$. Hence,  $\xi_{\gamma}$ lies in $\mathcal{Y}$, i.e., the last $(d-n)$ components of $(w_0+\tilde{w})$ vanish.  
We claim that $\gamma P_1$, $\gamma_2 P_2\notin\mathcal{Y}$. The reason is simple. Assume that $\gamma P_1\in\mathcal{Y}$, then the geodesic $\gamma n_{u_1}A{\cdot}o$ must lie in $G_0/K_0$ since the two end points $\gamma P_1$ and $\lim_{r\rightarrow\infty}n_{u_1}a_r{\cdot}o$ both lie in $\mathcal{Y}$. This implies that $\gamma\in\Gamma_0$, whereas we have assumed that $\gamma\notin\Gamma_0$. The case $\gamma P_2$ is shown in the fashion.
As a consequence, the geodesic $\gamma C_u$ either intersects $G_0/K_0$ or is away from $G_0/K_0$, i.e., $\delta_u(\gamma\gamma_2)>1$.
The first case indicates that 
the (minimal) distance between $\gamma\gamma_2C_u$ and the submanifold $G_0/K_0$ is achieved at a 
finite $r$, hence
$M(\gamma^{\prime})
N_u(\gamma^{\prime})\ne 0$ (otherwise the distance will be achieved at $r=0$ or $\infty$). By Lemma \ref{c2}, the second case directly indicates $M(\gamma
\gamma_2)
N_u(\gamma
\gamma_2)\ne 0$ since $\delta_u(\gamma
\gamma_2)>1$ in this situation. 
Replacing $\gamma$ with $\gamma\gamma_2$, we are done.
\item
If $N_u(\gamma)=0$ for some $u\in\mathbb{R}^{n-1}$, then $M(\gamma)\ne 0$ by Lemma \ref{mnz}, and $f_{\gamma}(u,r)\rightarrow 1=\delta_u(\gamma)$ as $r\rightarrow 0^+$. Due to the same reason for the case $M(\gamma)=0$, $\xi_{\gamma}(u)$ lies
in $\mathcal{Y}$, i.e., the last $(d-n)$ components of $(w_0+\mathring{w})$ vanish. Meanwhile, $\xi_{\gamma}$
does not lie in $\mathcal{Y}$, otherwise
we have $\gamma C_u\subset G_0/K_0$ which 
implies that  $\gamma\in\Gamma_0$.
As $\gamma_2$ varies, $u_1$ and $u_2$ can be arbitrarily large. If there exist $u_1$, $u_2$ such that $\gamma P_1$, $\gamma P_2$ do not lie in $\mathcal{Y}$, then the geodesic $\gamma\gamma_2 C_u$ meets or stays away from $G_0/K_0$. 
Both cases indicates that 
$M(\gamma\gamma_2)
N_u(\gamma\gamma_2)\ne 0$ (see the argument for the case $M(\gamma)=0$). 
Let $\gamma^{\prime}=
\gamma\gamma_2$, then 
we are done.
If $\gamma P_1$ or $\gamma P_2$ lies in $\mathcal{X}$ for any $u_1$, $u_2$, then for
any $i\geqslant n$ we have $(w_0+\mathring{w})_i=0$ as $r\rightarrow 0$, where $\mathring{w}$ depends on $u_1$ or $u_2$ (not $u$).
This is equivalent to  $$\frac{\alpha_i-\frac{u_{i+1,1}}{2}}{\frac{1+u_{11}}{2}+\beta}=-w_{0i},\quad \forall~i\geqslant n,$$
for any $u_1$ or $u_2$ (which correspond to $\gamma_2$). Assume that
the above identity holds for infinitely many $u_1$ which can be large enough. Then we get $-w_{0i}=\frac{u_{i+1,1}}{1-u_{11}}$ ($i\geqslant n$), the proportion of the coefficients of $|u|^2$ of $\alpha_i-\frac{u_{i+1,1}}{2}$ and $\frac{1+u_{11}}{2}+\beta$. The proportion of the constant terms of $\alpha_i-\frac{u_{i+1,1}}{2}$ and $\frac{1+u_{11}}{2}+\beta$ should also be equal to $-w_{0i}=\frac{u_{i+1,1}}{1-u_{11}}$. This yields
$\frac{u_{i+1,1}}{1+u_{11}}=
-\frac{u_{i+1,1}}{1-u_{11}}$ from which we get $u_{i+1,1}=0$ ($i\geqslant n$). It follows that
$w_{0i}=0$ ($i\geqslant n$) and $\xi_{\gamma}$ lies in $\mathcal{Y}$ as $(w_0+\tilde{w})_i=0$ for $i\geqslant n$, but we have shown that $\xi_{\gamma}$ does not lie in $\mathcal{Y}$.
\end{itemize}

Now we show the first two properties of Proposition \ref{x2}.
Let $\gamma=a_{r_0}n_{w_0}
k(\gamma)$ and write
$k(\gamma)\gamma_3=
n_wa_sk\in NAK$ where $\gamma_3=a_{\ell_0^p}
k_0^p\in AM_0$ ($p\in\mathbb{Z}$). Then
$\gamma\gamma_3=
a_{r_0s}n_{(w_0+w)/s}k$ for some $k\in K$. Applying (\ref{s1}) and (\ref{vs}) to $u=0$, we get the $i$-th component of $\frac{w_0+w}{s}$:
$$\left(\frac{w_0+w}{s}
\right)_i=
\frac{w_{0\,i}(1-u_{11})+u_{i+1,\,1}}{2}\,\ell_0^p+\frac{w_{0\,i}(1+u_{11})-u_{i+1,\,1}}{2}\,\ell_0^{-p}.
$$
Although   
$M(\gamma)$ is nonzero under our assumption, it is possible that certain $m_i$ ($i\geqslant n$) might be zero.
If $m_i=\frac{w_{0\,i}(1-u_{11})+u_{i+1,\,1}}{2}\ne 0$ and 
$\frac{w_{0\,i}(1+u_{11})-u_{i+1,\,1}}{2}\ne0$, then we can choose proper $p$ (i.e., $\gamma_3$) such that 
$\left|\left(\frac{w_0+w}{s}
\right)_i\right|$ is large enough. Here we require
that $p<0$ and $|p|$ is large. Replacing $\gamma$ with $\gamma\gamma_3$, we might as well assume that
$|w_{0i}|$ is large enough. 
If $\frac{w_{0\,i}(1-u_{11})+u_{i+1,\,1}}{2}=\frac{w_{0\,i}(1+u_{11})-u_{i+1,\,1}}{2}=0$, then $w_{0i}=u_{i+1,1}=0$. 
It follows from (\ref{nexp}) that $n_i=\sum_{j=2}^n
(u_{i+1,j}-w_{0i}u_{1j})u_{j-1}$. Thus, $|n_i|$ achieves its 
minimal value at $u=0\in\mathcal{F}$.
There is nothing to prove
in this case.
Next we discuss the case where
$m_i=0$ and $w_{0i}\ne0$. It follows that
$\frac{w_{0\,i}(1+u_{11})-u_{i+1,\,1}}{2}\ne0$ (otherwise,  $w_{0i}=0$). The above 
formula shows that 
$\left|\left(\frac{w_0+w}{s}
\right)_i\right|$ can be large enough provided that  $p<0$ and $|p|$ is large. 
Replacing $\gamma$ with 
$\gamma\gamma_3$ we assume that $|w_{0i}|$ is large enough. 
The case $m_i=0$ and 
$w_{0i}=0$ is equivalent to the case $\frac{w_{0\,i}(1-u_{11})+u_{i+1,\,1}}{2}=\frac{w_{0\,i}(1+u_{11})-u_{i+1,\,1}}{2}=0$ which has been discussed in above. In all, except the trivial cases (when $w_{0i}=u_{i+1,1}=0$) we can and will assume that
$|w_{0i}|$ is large enough. 
In view of Proposition \ref{u111}, this implies that both $|m_i|$ and $M(\gamma)$ are large enough (again, note that $M(\gamma)\ne0$, i.e., not all $m_i$ are zero).  
By (\ref{nexp}) we can write $n_i$ as $\sum\limits_{j=2}^{n}
n_{i,j-1}^{\prime}$ where
$$n_{i,j-1}^{\prime}=
m_iu_{j-1}^2+(u_{i+1,j}-w_{0i}u_{1j})u_{j-1}+\frac{w_{0i}(1+u_{11})-u_{i+1,1}}{2(n-1)}.$$
Each $n_{i,j-1}^{\prime}$ is a degree 2 polynomial with
leading coefficient being $m_i$. Consequently, 
the $u$ at which 
$N_u(\gamma)$ achieves its minimal value must lie nearby 
$u_0=(u_{0i})$ where 
$u_{0i}=-\frac{u_{i+1,j}-w_{0i}u_{1j}}{2m_i}$. If $|w_{0i}|$ is sufficiently large, then 
$u_{0i}$ is close to 
$\frac{u_{1j}}{1-u_{11}}$
which lies in a bounded interval in $\mathbb{R}$ (independent of $\gamma$ in view of Proposition \ref{u111}). Likewise, based on (\ref{mns}) we can show that the $u$ at which 
$2\sum_{i=n}^{d-1}
m_in_i$ achieves its minimal value also lies in a bounded domain in $\mathbb{R}^{n-1}$. We 
omit the details.  

Finally, we show the third property of Proposition \ref{x2}. Assume that there exist a sequence $\{\gamma_i\in\Gamma\smallsetminus\Gamma_0\,|\,\widetilde{\gamma}_i\ne\widetilde{\gamma}_j\,\,\textup{for}\,\,i\ne j\}$ and $u_i\in\mathbb{R}^{n-1} $ such that $M(\gamma_i)N_{u_i}(\gamma_i)\rightarrow 0$ as $i\rightarrow\infty$. It follows from (\ref{MNQ}) that $\delta_{u_i}(\gamma_i)\rightarrow 1$ as $i\rightarrow\infty$. Thus, $\delta_{u_{i\ast}}(\gamma)\rightarrow1$ where $u_{i\ast}$ denote 
the point $u$ at which $\delta_u(\gamma)$ achieves its minimal value.
Since $u_{i\ast}$ lies in a
fixed compact subset $\mathcal{F}\subset
\mathbb{R}^{n-1}$,
by passing to a subsequence if necessary, we assume that $\{u_{i\ast}\}$ converges to $u\in\mathcal{F}$.
Then $\delta_u(\gamma_i)
\rightarrow1$ as $i\rightarrow\infty$. This contradicts Corollary \ref{uag}.
\end{proof}

\section{$f$ and $k_f$}\label{fkf}
For the trace formula to be
valid, $f$ and $k_f$ should
fulfill some conditions (see the end of Sect.\,\ref{tr}).
In this section we
examine these conditions. Remember that $f(g)=\Phi_{\mu}\left({\rm d}_{G/K}(g{\cdot}
o,\,e{\cdot}o)\right)$ for $g\in G$. Hence $f$ is
bi-$K$-invariant and $f=f_K$ (see Sect.\,\ref{tr} for the
definition of $f_U$). 
Next we show $f\in C_{\textup{unif}}(G)$.
Let $U\subset G$ be a small enough compact
neighborhood of $e$ which is symmetric, i.e., $U^{-1}=U$ (such neighborhood exists:  take $U=V\cap V^{-1}$ where $V$ is a small neighborhood of $e$).
For any $h_1$, $h_2\in U$, we have
$$\big|{\rm d}_{G/K}(h_1gh_2{\cdot}o,\,e{\cdot}o)-
{\rm d}_{G/K}(g{\cdot}o,\,e{\cdot}o)\big|=
\big|{\rm d}_{G/K}(gh_2{\cdot}o,\,h_1^{-1}{\cdot}o)-
{\rm d}_{G/K}(g{\cdot}o,\,e{\cdot}o)\big|\leqslant
2\,{\rm diam}(U)$$
where ${\rm diam}(U)$
means the diameter of $U$ (modulo $K$), i.e., ${\rm diam}(U)
=\min\limits_{u_1,\,u_2\in U}{\rm d}_{G/K}(u_1{\cdot}o,u_2{\cdot}o)$. By the definition of $f$, one has
$$\frac{f_{U}(g)}{f(g)}\leqslant
e^{2\mu\,{\rm diam}(U)}.$$
Hence, $f_{U}$ is integrable if $f$ is so. The latter  
can be shown by the following
computation 
\begin{eqnarray*}
\int_{G}f(g)dg
&=&
\int_N\int_A\Phi_{\mu}
\left({\rm d}_{G/K}(na{\cdot}o,\,e{\cdot}o)\right)dadn\\[0.2cm]
&=&2^d\left(\sqrt{\frac{\pi}{2\mu}}\right)^{d-1}K_{\frac{d-1}{2}}(\mu)<\infty
\end{eqnarray*}
where the second step is a copy
of the computation for $h_f(\lambda_j)$,  dropping the term $\eta_{\nu_j}$ thereof (see Sect.\,\ref{spec}). 
As a result, $f\in C_{\textup{unif}}(G)$. 
For the rest of this section we check the locally uniform convergence of $k_f$. The
$L^{\infty}$-norm of any Laplace eigenfunction $\phi_j$ on $X$ satisfies the classical H\"{o}rmander's
bound (see \cite{hoe}): $$\sup_{x\in X}\,|\phi_j(x)|\leqslant
C\,\lambda_j^{(d-1)/4}\|\phi_j\|_{L^2(X)}$$ where $\lambda_j$ is
the Laplace eigenvalue of $\phi_j$, $C$ is uniform for all $j$.
In our context, $\phi_j$'s are orthonormal basis of $L^2(X)$, so we have:
$\sup_{x\in X}\,|\phi_j(x)|\leqslant
C\,\lambda_j^{(d-1)/4}$. 
For the convergence of $k_f$, it suffices to
consider those $\phi_j$'s with large eigenvalues, i.e., $\nu_j\in\,i\,\mathbb{R}$. Thus, we write
$\nu_j=i\,r_j$ with $r_j\in\mathbb{R}_+$. Then
$\lambda_j=\left(\frac{d-1}{2}\right)^2+r_j^2$. Substituting $v=ir_j$, $x=1$ into the following
formula (see 8.432.5 of \cite{gr})
$$K_{\nu}(xz)=\frac{\Gamma\left(\nu+\frac{1}{2}\right)(2z)^{\nu}}{x^{\nu}\Gamma\left(\frac{1}{2}\right)}\int_{0}^{\infty}\frac{\cos\,xt\,dt}{(t^2+z^2)^{\nu+\frac{1}{2}}},\quad\textup{Re}\left(\nu+\frac{1}{2}\right)\geqslant0,\,\,x>0,\,\,|\textup{arg}\,z|<\frac{\pi}{2},$$
we get
$$K_{ir_j}(z)=\frac{\Gamma\left(
1/2+ir_j\right)}{\Gamma(1/2)}\,(2z)^{ir_j}
\int_{0}^{\infty}
\frac{\cos t\,dt}{\left(t^2+z^2\right)^{1/2+ir_j}},\quad z>0.$$
The integration by parts shows that
$$\int_{0}^{\infty}
\frac{\cos t\,dt}{\left(t^2+
z^2\right)^{1/2+ir_j}}=
(1+2ir_j)\,\int_{0}^{\infty}\frac{t\,\sin t\,dt}{\left(t^2+z^2\right)^{3/2+
ir_j}}.$$
The integral on the right hand side of the above equality 
absolutely converges for $z\in\mathbb{R}\smallsetminus\{0\}$. Thus, $K_{ir_j}(z)$ is uniformly upper bounded by
$\left|\Gamma\left(\frac{1}{2}+ir_j\right)\right| r_j$ for all
$z\in\mathbb{R}\smallsetminus\{0\}$. By the Stirling formula on Gamma function (where $a$,
$b\in\mathbb{R}$):
$$|\Gamma(a+ib)|=\sqrt{2\pi}\,|b|^{a-1/2}e^{-|b|\,\pi/2}\left[1+\mathcal{O}
\left(\frac{1}{|b|}\right)\right],\quad\textup{as}\,\,|b|\rightarrow\infty, $$
we get a bound:
$K_{ir_j}(\mu)=\mathcal{O}\left(r_je^{-\frac{\pi}{2}r_j}\right)$.
Combining this bound with H\"{o}rmander's bound, we have:
$$K_{ir_j}(\mu)\phi_j(z)\overline{\phi_j(w)}=\mathcal{O}\left(r_je^{-\frac{\pi}{2}r_j}\,\left[r_j^{(d-1)/2}\right]^2\right)=\mathcal{O}\left(r_j^{d}e^{-\frac{\pi}{2}r_j}\right)\quad\textup{as}\quad j\rightarrow\infty.$$
The spectrum $\{\lambda_j\}$ of the Laplacian is discrete with
$\infty$ as the unique accumulation point and each eigenvalue
$\lambda_j$ occurs with finite multiplicity, so is
$\{r_j\in\mathbb{R}_+\}$. Let $N(x)$ be the counting function of
Laplace eigenvalues with multiplicities over $X$:
$$N(x):=\sum\limits_{\phi_j:\,r_j\leqslant x}1.$$
Weyl's law\index{Weyl's law} gives the asymptotic\index{$N(x)$}
of $N(x)$ for large $x$ (see \cite{mp}).
\begin{equation*}
N(x)=\frac{\textup{vol}(X)}{(4\pi)^{d/2}\Gamma\left(\frac{d}{2}+1\right)}\,x^d+o\left(x^d\right),\quad\textup{as}\,\,x\rightarrow\infty.
\end{equation*}
Let $A_j=\sqrt{\lambda_j}-r_j$, then $A_j=\mathcal{O}\left(r_j^{-1}\right)$ as $j\rightarrow\infty$. With the bound on
$K_{ir_j}(\mu)\phi_j(z)\overline{\phi_j(w)}$ obtained in above and
the formula
$h_f(\lambda_j)=2^d\left(\sqrt{\frac{\pi}{2\mu}}\right)^{d-1}
K_{\nu_j}(\mu)$, we have:
$$k_f\ll\sum\limits_{j}
r_j^de^{-\frac{\pi}{2}r_j}<\sum\limits_{j}
\lambda_j^{\frac{d}{2}}e^{-\frac{\pi}{2}\left(\sqrt{\lambda_j}-A_j\right)}\,\asymp\sum\limits_{j}
\lambda_j^{\frac{d}{2}}e^{-\frac{\pi}{2}\sqrt{\lambda_j}}=
\int_{0}^{\infty}
x^{\frac{d}{2}}e^{-\frac{\pi}{2}\sqrt{x}}dN(x).$$ Here $dN(x)$\index{$dN(x)$} means the measure on $\mathbb{R}_+$ with mass $1$ at $x=r_j$ (multiplicities counted), otherwise $0$.
Integration by parts shows that
$$\int_{0}^{\infty} x^{\frac{d}{2}}e^{-\frac{\pi}{2}\sqrt{x}}dN(x)=
x^{\frac{d}{2}}e^{-\frac{\pi}{2}\sqrt{x}}N(x)\Big|_{0}^{\infty}-
\int_{0}^{\infty}e^{-\frac{\pi}{2}\sqrt{x}}\left(\frac{d}{2}x^{\frac{d}{2}-1}-\frac{\pi}{4}x^{\frac{d-1}{2}}\right)N(x)dx.$$
Applying Weyl's law on $N(x)$ to the right hand side of this identity, we see that 
the integral on the left hand side converges. Thus, we have shown the absolute and locally uniform
convergence of $k_f$.
\begin{ac*}
The author would like to thank Professor A. Deitmar for introducing him this topic
as well as many helpful discussions.\end{ac*}

\noindent {\it Author's address}: {Depart. of Mathematics, Bar-Ilan Univ., Ramat-Gan 52900, Israel}\\
{\it Author's email}: {\tt fsu@\,math.biu.ac.il}
\end{document}